\documentclass[a4paper,11pt]{amsart}

\usepackage{amsfonts,amsmath,mathrsfs,amssymb,amsthm,amscd,latexsym,amstext,amsxtra}
 \usepackage[usenames,dvipsnames,svgnames,table]{xcolor}

\usepackage[left=2cm,right=2cm, top=3cm, bottom=3cm]{geometry}

\usepackage[all,cmtip]{xy}
\usepackage{enumerate}
\xyoption{all}
\usepackage{srcltx} %per fer cerca inversa
\usepackage{tikz}
\usepackage{tikz-cd}
\usepackage{textcomp}

\usepackage{amsthm, amsmath, amssymb,latexsym} 
\usepackage{amsfonts,epsfig,amscd}
\usepackage[]{fontenc}
\usepackage{enumerate}
\usepackage[all]{xy}
\usepackage{mathrsfs}

\newtheorem{theorem}{Theorem}[section]
\newtheorem{lemma}[theorem]{Lemma}
\newtheorem{corollary}[theorem]{Corollary}
\newtheorem{proposition}[theorem]{Proposition}
\newtheorem{remark}[theorem]{Remark}
\newtheorem{definition}[theorem]{Definition}

\newcommand{\nc}{\newcommand}
\newcommand{\cH}{{\mathcal H}}
\newcommand{\cA}{{\mathcal A}}

\newcommand{\cC}{{\mathcal C}}

\newcommand{\cO}{{\mathcal O}}

\newcommand{\cK}{{\mathcal K}}
\newcommand{\cX}{{\mathcal X}}

\newcommand{\cF}{{\mathcal F}}

\newcommand{\cQ}{{\mathcal Q}}

\newcommand{\cP}{{\mathcal P}}
\newcommand{\cL}{{\mathcal L}}
\newcommand{\cM}{{\mathcal M}}

\newcommand{\cU}{{\mathcal U}}

\newcommand{\bH}{{\mathbb H}}

\newcommand{\bC}{{\mathbb C}}

\newcommand{\bZ}{{\mathbb Z}}

\newcommand{\bN}{{\mathbb N}}
\newcommand{\bP}{{\mathbb P}}

\newcommand{\bU}{{\mathbb U}}

\newcommand{\bR}{{\mathbb R}}

\nc{\OLD}{\textcolor{blue}{OLD. (}}
\nc{\NEW}{\textcolor{blue}{NEW. (}}
\nc{\blue}[1]{\textcolor{blue}{ #1}}

\nc{\fr}{{\rightarrow}}
\nc{\co}{{\nabla}}

\nc{\cu}{{\overlineline{\nabla}}}
\nc{\gmc}{\nabla}
\nc{\mtin}[1]{\mbox{{\tiny #1}}}
\nc{\rank}[1]{r_{\mbox{{\tiny #1}}}}
\DeclareMathOperator{\Ext}{Ext}
\DeclareMathOperator{\Hom}{Hom}

\DeclareMathOperator{\rk}{rk}

\DeclareMathOperator{\spec}{Spec}
\DeclareMathOperator{\Pic}{Pic}

\newtheoremstyle{dico}% name of the style to be used
 {\baselineskip}   % ABOVESPACE
  {\topsep}   % BELOWSPACE
  {}  % BODYFONT
  {0pt}       % INDENT (empty value is the same as 0pt)
  {} % HEADFONT
  {.}         % HEADPUNCT
  {5pt plus 1pt minus 1pt} % HEADSPACE
  {}          % CUSTOM-HEAD-SPEC
\theoremstyle{dico}

\numberwithin{equation}{section}

\newcommand{\om}{\omega}

\newcommand{\Cliff}{\operatorname{Cliff}}
\newcommand{\gon}{ \operatorname{gon}}

\nc{\lra}{{\longrightarrow}}
\newcommand{\vhs}{\bH_\bZ}

\newcommand{\Sym}{Sym}
\newcommand{\cliff}{cliff}
\newcommand{\Nm}{Nm}

\newcommand{\Div}{Div}
\newcommand{\St}{St}

\makeatletter 
\@addtoreset{equation}{section}
\makeatother

\begin{document}

\title{On the  Jacobian locus in the Prym locus and geodesics}

\author[S. Torelli]{Sara Torelli}
%\author{Sara Torelli}
\address{Dipartimento di Matematica,
	Universit\`a di Pavia,
	Via Ferrata, 5,
	27100 Pavia, Italy}
\email{sara.torelli7@gmail.com}

\thanks{The authors were supported by PRIN 2017 Moduli spaces and Lie
  Theory, INdAM - GNSAGA,
FAR 2016 (Pavia) ``Variet\`a algebriche, calcolo algebrico, grafi orientati e topologici'' and MIUR, Programma Dipartimenti di Eccellenza
(2018-2022) - Dipartimento di Matematica ``F. Casorati'', Universit\`a degli Studi di Pavia.}
 
%\MSC{}

 
 %%www.ams.org/msc
%  14C30 =  Transcendental methods, Hodge theory
%  14H10 =  Families, moduli (algebraic)
%  14H15 =  Families, moduli (analytic) 
%  32G20 =  Period matrices, variation of Hodge structure; degenerations

\date{\today}

\begin{abstract}
	In the paper we consider the Jacobian locus $\overline{J_g}$ and the Prym locus $\overline{P_{g+1}}$, in the moduli space  $A_g$ of principally polarized abelian varieties of dimension $g$, for $g\geq 7$, and we study the extrinsic geometry of $\overline{J_g}\subset \overline{P_{g+1}}$, under the inclusion provided by the theory of generalized Prym varieties as introduced by Beauville. More precisely, we study certain geodesic curves with respect to the Siegel metric of $A_g$, starting at a Jacobian variety $[JC]\in A_g$ of a curve $[C]\in M_g$ and with direction $\zeta\in  T_{[JC]}J_g$. We prove that for a general $JC$, any geodesic of this kind is not contained in $\overline{J_g}$ and even in $\overline{P_{g+1}}$, if $\zeta$ has rank $k<\Cliff C-3$, where $\Cliff C$ denotes the Clifford index of $C$.
\end{abstract}

\maketitle

\tableofcontents

\section{Introduction} The study of the extrinsic geometry of the Jacobian locus $\overline{J_g}$ in the moduli space $A_g$ of principally polarized abelian varieties, namely of the closure of the locus of Jacobian varieties in $A_g$, has been stimulated by different many problems throughout its history (two of the best known are the Torelli and Schottky problems). In particular, the theory of generalized Prym varieties as developed by Beauville in \cite{B77}  highlighted a natural inclusion of the Jacobian locus inside boundary of the Prym locus $\overline{P_{g+1}}$, which is the closure in $A_g$ of the locus of Prym varieties associated to  \'etale coverings of degree $2$ onto smooth projective curves of genus $g+1$. In terms of the study of the extrinsic geometry of $\overline{J_g}\subset A_g$ such a theory suggested a refined study of $\overline{J_g}\subset \overline{P_{g+1}}$ by using the parametrization of Jacobian varieties as generalized Prym varieties.  In this paper we address such a problem by studying geodesics along the Jacobian locus, defined with respect to the metric induced on $A_g$ by the Siegel metric of the Siegel upper half space $H_g$ (which is nothing but that the universal covering of $A_g$). The precise questions we aim to answer are the following. Consider a  geodesic (even local) in $A_g$ and assume that it is tangent at a point of $J_g$ (that is, at a point not lying in the boundary). 

{\bf Question 1.} Are these geodesics of $A_g$ that are tangent at a point of $J_g$ locally contained in the Jacobian locus around this point? 

In case of negative answer, 

{\bf Question 2.} Is there any proper sub-locus than $A_g$ where these geodesics that are tangent at a point of $J_g$ are locally contained around this point? 

The motivating interest has been suggested by the fact that the Jacobian locus is expected to be rather curved inside $A_g$, meaning that not many geodesics as before  are expected to be (locally) contained inside it and even in any proper sub-locus of $A_g$. This fact, according to some classical results as in \cite{G_ATrascendental_1971} and even very recent as in \cite{CF19} and \cite{KO19}. Nevertheless, if there exists a locus smaller than $A_g$ and closed with respect to this local property on geodesics, then the first one that should be detected is the Prym locus, according to the result of \cite{B77} recalled above. The main result of the  paper gives a partial answer to these questions and shows that the expectation is true if we consider the Prym locus as the sub-locus of $A_g$ and certain geodesics satisfying a property depending on the clifford index of the curve that realizes the investigated point of $J_g$. 

\vspace{1.7mm}
Let $M_g$ be the moduli space of curves of genus $g$ and let $R_{g+1}$ be the moduli space of degree $2$ \'etale coverings of curves of genus $g+1$. For any $[C]\in M_g$, let $JC$ be the Jacobian variety and let  $\Cliff C$ denote the Clifford index of $C$ (see \cite{ACGH}). For any $\zeta \in T_{[JC]}J_g$, let $\xi \in H^1(C, T_C)$ such that $\zeta $ can be identified with the cup product map $\cup \xi: H^0(C, \omega_C)\to H^0(C, \omega_C)^\vee$, under the isomorphism  $T_{[JC]}A_g\simeq \Sym^2H^0(C, \omega_C)^\vee$. We define the rank of $\zeta$ as the rank of $\cup \xi$. 
 We prove the following
\begin{theorem}\label{Thm-Main1} Let $JC$ be a general Jacobian variety of dimension $g\geq 7$.  Then for any $\zeta\in  T_{[JC]}J_g$ of rank $k=\rk \zeta <\Cliff C-3$, the local geodesic in $A_g$ at $[JC]$ with direction $\zeta$ (defined with respect to the Siegel metric) is not contained in the Prym locus $\overline{P_{g+1}}$  (in particular, also in $\overline{J_g}\subset \overline{P_{g+1}}$).
	\end{theorem}

The proof of Theorem \ref{Thm-Main1} is obtained by working on families of abelian varieties over a smooth complex curve, containing the investigated geodesic of starting data $(JC, \zeta)$.  One proves separately that such a geodesic is not locally contained in $\overline{J_g}$ (Jacobian case) and then even in $\overline{P_{g+1}}$ (Prym case). The argument in the {\em Jacobian case} is a straightforward application of some previous works: especially from \cite{GPT18}, but some preparatory results are contained in \cite{PT}, \cite{GST} and \cite{GT19} (see Section \ref{sec:Jacobian}, Lemma  \ref{lem:Jacobians}). 
%   In both cases (Jacobian / Prym), one considers a family of curves naturally associated to the family of abelian varieties and reads off the piece of information on the rank among geodesics by introducing a suitable diagram depending on the family of curves and on a line bundle. This splits the proof in the subcase where such a line bundle contributes to the Clifford index of the general fibre and when not. 
 The argument in the {\em Prym case} is much more complicated and new, involving degeneration techniques (see Section \ref{sec:draft}). Just to give an idea, in this case there is a family of Prym varieties degenerating to a Jacobian variety, which is realized as a generalized Prym variety as in \cite{B77}. So in this case one must take into account two things. The first one is that a priori there are different ways to realize a Jacobian variety as a generalized Prym variety, even for a general Jacobian (see \ref{cor:genJac} and also \ref{lem:normJacPrym}). The second one is that the limit curve in the corresponding family of curves is no longer realizing the limit Jacobian but it is only involved in the construction of the covering realizing the limit generalized Prym variety and it might be even nodal and not irreducible.  Consequently, an estimation of the Clifford index in terms of the curve realizing the limit Jacobian requires to use the theory of admissible coverings as in \cite{HM82} to study the limit gonality (see \ref{lem:cliffnorm} together with \ref{lem:normJacPrym}). 
%Concerning the subcase where the line bundle does not contribute to the Clifford index,  there is a common subcase (case $2$, Jacobian  / Prym) studied by applying techniques developed in \cite{PT, GST, GPT18} (or kind of them adapted to the Prym case) on the rank of a suitable subbunble of the Hodge  bundle along geodesics (namely, using certain tools of complex Hodge  theory).  An additional case appears dealing with Pryms, depending on the family of coverings  (case $3$) and leading to ask more generality to the starting datum $JC$. 
 As a remark, the structure of the proof seems to work similarly to study geodesics satisfying similar conditions of tangency to other loci in $A_g$, defined as the image of boundary divisors of the moduli of curves. But this requires for sure a more sophisticated techniques to compute the Clifford index (see e.g. \cite{C11}).
%The proof of theorems \ref{Thm-Main2} and \ref{Thm-Main3} are obtained by analysing the geometry of the locus $Z$. 
%The proof of Theorem \ref{Thm-Main3} is a straight forward application of the previous one , just specifying the condition on the Cfifford index to the case of Shiffer variations.
  
 \vspace{1.5mm}
The plan of the paper is the following. In \S \ref{sec:prel} we collect some preliminaries and working lemmas. More precisely, in Subsection \ref{sec:JacPrym} we focus on degenerations of Prym to Jacobian varieties; in Subsection \ref{sec:acgc} we relate the limit gonality of admissible coverings with respect to the Clifford index of a family of curves; in Subsection \ref{sec:vhs} we recall some tools in Hodge  theory and in Subsection \ref{sec:isotropic} some on certain related isotropic subspaces.
% Lemma \ref{lem:normJacPrym} concerning generalized Prymvarieties defining a general Jacobian, Lemma \ref{lem:cliffnorm} comparing the limit gonality, . The main references are \cite{B77,F11,HM82, C85, PT,GST, GPT18}. 
 In \S
\ref{sec:Jacobian} we prove the so called Jacobian case;  in
\S \ref{sec:draft} the so called Prym case, ending the proof of  Theorem \ref{Thm-Main1}.
% and  \ref{Thm-Main3}.

\vspace{1.5mm}
{\bfseries \noindent{Acknowledgements}}.  I am very grateful to  Gian Pietro Pirola, not only for any discussion about the topics, but above all for having always supported me in writing the paper. I would also like to thank Joan Carles Naranjo and Alessandro Ghigi for their useful explanation and comments.
%also Alessandro Ghigi, Victor Gonz{\'a}lez-Alonso and Lidia Stoppino, co-authors of some previous works where some useful techniques have been developed.
\section{Preliminaries}\label{sec:prel}
\subsection{Jacobian and Prym locus in the moduli space of abelian varieties}\label{sec:JacPrym} Assume $g\geq 7$. 
Let $M_g$ be the moduli space of curves of genus $g$ and let $A_g$ be the moduli space of principally polarized abelian varieties. 
The Torelli morphism 
\begin{equation}j: M_g\to A_g , \quad [C]\mapsto [JC]\end{equation}
sends a genus $g$ smooth projective curve $C$ to its Jacobian variety $JC$ (up to isomorphism). Its image $J_g=j(M_g)\subset A_g$ defines a proper locus for $g>3$ and its closure $\overline{J_g}\subset A_g$ is called the Jacobian locus. By using the isomorphisms 
$$T_{[C]}M_g\simeq H^1(C,T_C), \quad  T_{[JC]}A_g\simeq \Sym^2H^0(C, \omega_C)^{\vee},$$ the differential of the Torelli map 
$$dj:H^1(C,T_C)\to  \Sym^2H^0(C, \omega_C)^{\vee},\quad \xi \mapsto \zeta$$
is given by $\zeta=\cup\xi:H^0(C, \omega_C)\to H^0(C,\omega_C)^\vee,$ the cup product map with $\xi.$

\vspace{1.5mm}
Let $R_{g+1}$ be the moduli space parametrizing $2$-sheeted \'etale coverings $\pi: \tilde{C}\to C'$ between smooth projective curves of genus $\tilde{g}=g(\tilde{C})=2g(C')-1$ and $g'=g(C')=g+1$, respectively (modulo isomorphism). A point in $R_{g+1}$ is an isomorphism class assigned equivalently by \begin{itemize}
		\item[(i)] a pair $(\tilde{C}, i)$, where $i:\tilde{C}\to \tilde{C}$ is an involution such that $\tilde{C}/(i)=C'$,
		\item[(ii)] a pair $(C', \eta)$, where  $\eta \in Pic^0(C')\setminus \{\cO_{C'}\}$ is a $2$-torsion point  and $\tilde{C}= \spec (\cO_C\oplus \eta)$.
		\end{itemize}
		One can reconstruct the data $(i)$ from $(ii)$ and conversely (see \cite{B77}, \cite{M74}) and so with a little abuse of notations we will describe a point of $R_{g+1}$ in both ways depending on the setting.
The Prym morphism 
$$Pr: R_{g+1}\to A_g, \quad  [(C', \eta)]\mapsto [P(C',\eta)]$$ 
maps a pair $(C', \eta)$ to its Prym variety $P(C',\eta)$(up to isomorphisms). 
By definition, the Prym variety is $P(C',\eta)=\ker \Nm^0,$ namely the connected component containing zero in the kernel of the norm map $\Nm: J\tilde{C}\to JC'$, and it is a principally polarized abelian variety of dimension $g$, with the principal polarization given by one half of the restriction of that on $J\tilde{C}$. 
The image $P_{g+1}=Pr(R_{g+1})\subset A_g$ of the Prym morphism defines a proper locus for $g\geq 6$ and its closure $\overline{P_{g+1}}$ is called the Prym locus.
By construction, the forgetful functor $\pi_{R_{g+1}}: R_{g+1}\to M_{g+1},$ sending $[(C',\eta)]\mapsto [C']$, is an \'etale finite covering and so we can identify  $T_{[(C', \eta)]}R_{g+1}\simeq T_{[C']}M_g.$ Using the isomorphisms $$T_{[C', \eta]}R_{g+1}\simeq H^1(C', T_{C'}), \quad T_{[P(C, \eta)]}A_g\simeq \Sym^2H^1(C',\eta)\simeq \Sym^2 H^0(C', \omega_{C'}\otimes \eta)^{\vee},$$ the differential of the Prym map  $$d Pr:H^1(C', T_{C'})\to \Sym^2H^1(C',\eta), \quad \xi' \longmapsto \zeta',$$
is given by $\zeta'=\cup\xi':H^0(C', \omega_{C'}\otimes \eta)\to H^0(C', \omega_{C'}\otimes \eta)^{\vee}$, the cup product with $\xi'.$

\smallskip
\subsubsection{The Jacobian locus in the Prym locus}
We are interested in relating the Jacobian locus $\overline{J_g}$ and the Prym locus $\overline{P_{g+1}}$ in $A_g$. 
%These are respectively defined as $\overline{J_g}$, the closure of $J_g=j(M_g)$ in $A_g$, and $\overline{P_{g+1}}$, is the closure of $P_{g+1}=Pr (R_{g+1})$ in $A_g$ and for $g\geq 7$ define proper sublocus of $A_g$.
%\blue{sistemare quando decidi come vedere gli spazi, and the Torelli morphism is finite  and Prym morphism is generically finite.} 
By Beauville \cite{B77} (or previously by Wirtinger for an analytic proof), the Jacobian locus in genus $g$ fits inside the boundary of the Prym locus in genus $g+1$.  
To formalize this properly, we recall the definition of a generalized Prym variety and the extended Prym map as given in \cite[section 3,5,6]{B77}.
\begin{definition}\label{def:genPrym} Let $\tilde{C}$ be a connected curve of genus $2g+1$ with ordinary double points, let $i: \tilde{C}\to \tilde{C}$ be an involution, let $C'$ be the quotient curve, let $\Nm: J\tilde{C}\to JC$ be the norm map and let $P=\ker \Nm^0$ be the associated group variety. Assume $(\ast\ast)$:
	\begin{itemize}
		\item[1.] $i$ is not the identity on any component of $\tilde{C}$;
		\item[2.] $p_a(C')=g+1$;
		\item[3.] $P$ is an abelian variety.
		\end{itemize}
		In this case, we call $P$ {\em the generalized Prym variety} of $(\tilde{C}, i)$. If furthermore $\tilde{C}$ is smooth, we say that $P$ {\em is a standard Prym variety}.
	\end{definition}
	Set $r$ the number of fixed non singular points of $(\tilde{C},i)$; $n'_f$ the number of nodes of $\tilde{C}$ fixed under $i$, with the 2 branches exchanged; $2n_e$ the number of nodes exchanged under $i$ and $2c_e$ the number of components exchanged under $i$. 
	\begin{remark}\label{rem-type*} The condition $(\ast \ast)$ is equivalent to $r=n'_f=0$ and $c_e=n_e$ (\cite[Lemma 5.1]{B77}) and for $n_e=c_e=0$ we recover generalized Prym varieties of type $(\ast)$ of \cite[section 3]{B77} (i.e. the fixed points of $i$ are exactly the singular points and at a singular point the two branches are not exchanged under $i$).
		\end{remark}
	There is an irreducible subvariety $\overline{S}\subset \overline{R_{g+1}}$ that carries a family of $q: \tilde{\cC}\to \overline{S}$ of stable curves and a $\overline{S}$-involution $i:\tilde{\cC}\to \tilde{\cC}$ such that 
	\begin{itemize}
		\item[a.] for any $s \in\overline{S}$, the induced involution $i_s:\tilde{\cC}_s\to \tilde{\cC}_s$ is different from the identity on each component of $\tilde{\cC}_s$;
		\item[b.] $\tilde{\cC}_s$ has genus $2g+1$, and the quotient curve $\tilde{\cC}_s/(i_s)$ has genus $g+1$;
		\item[c.] for any non-singular curve $\tilde{C}$ of genus $2g+1$ with a fixed point free involution $t$, the pair $(\tilde{C}, t)$ is isomorphic to $(\tilde{\cC}_s, i_s)$ for some $s\in \overline{S}$.
		\end{itemize}
	We set $\cC'=\tilde{\cC}/(i).$ The norm $\Nm:\Pic^0(\tilde{\cC}/\overline{S})\to \Pic^0(\cC'/\overline{S})$ defines $\overline{\mathcal{P}}=\ker \Nm^0$, which is a smooth algebraic space over $\overline{S}$. The set $S\subset \overline{S}$ of points where $\overline{\mathcal{P}}_s$ is an abelian variety is open and on this set $\overline{\mathcal{P}}_s$ is the Prym variety associated to $(\tilde{\cC}_s, i_s)$. The general fibre $(\tilde{\cC}_s, i_s)$ defines a generalized Prym variety as in Definition \ref{def:genPrym}. Denoting $\mathcal{P}$ the restriction of $\overline{\mathcal{P}}$ to $S$, we get a morphism $p: S \to A_g$
		whose general point is a generalized Prym variety as in Definition \ref{def:genPrym}. Furthermore, $S\subset  \overline{R_{g+1}}$, $S^0=S\cap R_{g+1}\subset S$ is a proper sublocus and $p_{|S^0}=Pr$, namely it is the Prym morphism $Pr$.

		We now focus on the inclusion $\overline{J_g} \subset \overline{P_{g+1}}$ provided by using generalized Prym varieties as above. First of all,   \begin{equation}\label{eq:JacinPrym}\overline{J_g}\subset p(S)\subset \overline{P_{g+1}}.
			\end{equation}
			This is shown by the following 
			
			{\em Construction} (see Figure 1).
			For any $[JC]\in J_g$, we can take $\tilde{C}=C_1\cup C_2$ with $C_1=C_2=C$ intersecting tranversally in a pair of distinct points $C_1\cap C_2=\{p_1, p_2\}$ and an involution $i:\tilde{C}\to \tilde{C}$ exchanging $C_1$ with $C_2$ and $p_1$ with $p_2$.  So the quotient $C'=\tilde{C}/(i)$ is an irreducible nodal curve with only one node $q$ whose normalization is $\nu: C\to C'$ and the covering $\pi: \tilde{C}\to C'$ is e\'tale of degree $2$ mapping the two nodes $p_1, p_2$ to $q$. Summing up, the pair $(\tilde{C},i )$ defines a generalized Prym variety as in Definition \ref{def:genPrym} with $n_e=c_e=1$ ( \cite[Theorem 5.4]{B77} or directly Wintinger's construction). 
			
			\smallskip
			We can be even more precise. Indeed, we can observe that $[C']\in \overline{M_{g+1}}$ in the construction above lies in the boundary divisor $\Delta_0\subset \overline{M_{g+1}}$, defined as the closure of the locus of the irreducible  one-nodal curves. Using the forgetful functor $\overline{\pi}_R:\overline{R_{g+1}}\to \overline{M_{g+1}}$, the pullback 
			$\overline{\pi}_{R}^{*}(\Delta_0)\subset \overline{R_{g+1}}$ has a subdivisor $\Delta^{''}_0$, defined as the closure of the locus of pairs $(C'_\alpha,\eta_\alpha)$ such that the normalization $\nu:C_\alpha\to C'_\alpha$ does not contain any exceptional curve, $\eta\neq \cO_{C'_\alpha}$ but $\nu^\ast\eta\simeq \cO_{C_\alpha}$. The unique pair $(C', \eta)$ corresponding to $(\tilde{C}, i)$ as above is such that $(C', \eta)\in \Delta_0^{''}$.
			 So we can more precisely describe $\overline{J_g}\subset \overline{P_{g+1}}$ as  
			 \begin{equation}\label{eq:Jac} \overline{J_g}\subset p(\Delta^{''}_0)\subset \overline{P_{g+1}}
			 \end{equation}
			 (See e.g. \cite[section 6, Example 6.5]{F11} for more details). 
			 
			 \smallskip
			 We end this section proving a lemma, which we will need in the proof of the main theorem. 
		\begin{lemma}\label{lem:normJacPrym}
			Let $P$ be a generalized Prym variety as in Definition \ref{def:genPrym} such that $n_e=c_e>0$. Assume that $P=JC$ for $[C]\in M_g$. Then for any $(\tilde{C}_s, i_s)$ such that $P= p(s)$, we have $\tilde{C}_s= A_s\cup i_s(A_s)$ with $A$ and $i(A)$ intersecting in points, $\#(A\cap i(A))\geq4$, and either $ A_s=C$ or $A_s= C\cup R$, where $R$ is a rational curve. Furthermore, $C$ is contained in the normalization of $\tilde{\nu}_s:\tilde{C}_{\tilde{\nu}s}\to\tilde{C}_s$ and of $\nu'_s:C'_{\nu s}\to C'_s=\tilde{C}_s/(i_s)$. 
			\end{lemma}
			\begin{proof} Fix $(\tilde{C}, i)=(\tilde{C}_s, i_s)$ such that $P= p(s)$ and $n_e=c_e>0$. Then $\tilde{C}= A\cup A'\cup B$, where $B$ is the union of the components of $\tilde{C}$ fixed by $i_s$, $A'=i(A)$ and $A$ and $A'$ have no common components. By \cite[Theorem 5.4]{B77},  it follows that  $B=\emptyset$. By contradiction, if $B\neq \emptyset$ by \cite[Theorem 5.4]{B77} we have $P\simeq JA \times Q$, where $Q$ is a principally polarized abelian variety, %(up to operating on $A$ as in  \cite[Lemma 5.1]{B77}) 
				and so $P$ is the product of principally polarized abelian varieties. Since by assumption $P=JC$, the only possibility is having $Q=\{0\}.$ We now use that, again by \cite[Theorem 5.4]{B77}, $Q$ is the Prym variety of $(B, {i}_{|B})$ and it might be of type $(*)$ because $B$ is fixed by $i$, then $\dim Q =p_a(B/(i))-1$ so both $p_a(B)=p_a(B/(i))=1$ and $i$ has no fixed points. Since $B$ must intersect $A$ in a fixed point, this case is not realizable. So we have $B=\emptyset$, by \cite[Lemma 5.1]{B77}, (i.e. up to pass to a subgraph of the graph of $\tilde{C}$) one can choose $A$ connected and tree-like (i.e. its graph is a tree and its irreducible components are all smooth), and by \cite[Theorem 5.4]{B77} $P\simeq JA$. Let $\nu_A: A_{\nu_A} \to A$ be the normalization of $A$. Then $P=JA=JA_{\nu_A}$, because $A$ is tree-like, by assumption $P=JC$  and so $JC\simeq JA \simeq JA_{\nu_A}$. Since $JA_{\nu_A}= \prod_{i\in I} JC_i,$ where $\{C_i\}_{i\in I}$ denotes the set of components of $A$, then $C$ must be one of the smooth components of $A_\nu$ and if any other exists, it must be a rational curves.  Consequently, we also have $C\subset A_{\nu_A}\subset \tilde{C}_{\tilde{\nu}}$. Since $C'= \tilde{C}/(i)= (A\cup A') / (i)$ and $A$ and $A'$ have no common components we conclude the same for $\nu': C'_\nu\to C'$. 
				\end{proof}
			
					\begin{corollary}\label{cor:genJac} Let $JC$ be a Jacobian variety and let $\cliff C >2$. Then any generalized Prym variety $P$ such that $P=JC$ is provided by Lemma \ref{lem:normJacPrym}.
					\end{corollary}
					\begin{proof} The proof follows by  \cite[Main theorem]{Sho84}, classifying generalized Prym varieties of type $n_e=c_e=0$ that occurs as Jacobians or products of Jacobians, among the complete list provided by \cite[Theorem 4.10]{B77} Prym varieties of type $n_e=c_e=0$. Indeed, let $(\tilde{C}, i)$ define $P=JC$. Assuming $g(C)\geq 7$, case (d) of \cite[Main theorem]{Sho84} is excluded and $C'= \tilde{C}/(i)$ must be one of the remaining cases: (a) it is either hyperelliptic, (b) it is is obtained fom a hyperelliptic curve by identifying two points; (iii) it is trigonal (meaning that it has a map of degree $2$ or $3$, respectively, over $\bP^1$). But now by lemma \ref{lem:normJacPrym}, the same must hold for $C$ since $C\subset C'_\nu$, where $\nu: C'_\nu \to C'$ is the normalization, and this is not admitted.
						\end{proof}
					
\begin{center}
\tikzset{every picture/.style={line width=0.40pt}} %set default line width to 0.75pt        
\begin{tikzpicture}[x=0.75pt,y=0.75pt,yscale=-1,xscale=1]
%uncomment if require: \path (0,318); %set diagram left start at 0, and has height of 318

%Curve Lines [id:da9379012943180767] 
\draw    (211,155) .. controls (118,105) and (119,57) .. (230,31) ;

%Curve Lines [id:da27215699007032423] 
\draw    (114,149) .. controls (189,115) and (236,72) .. (113,31) ;

%Curve Lines [id:da09266679638365027] 
\draw    (97,252) .. controls (150,259) and (227,207) .. (161,208) .. controls (95,209) and (179,257) .. (227,257) ;

%Straight Lines [id:da15929993720604596] 
\draw    (161,142) -- (161.96,193) ;
\draw [shift={(162,195)}, rotate = 268.92] [color={rgb, 255:red, 0; green, 0; blue, 0 }  ][line width=0.75]    (10.93,-3.29) .. controls (6.95,-1.4) and (3.31,-0.3) .. (0,0) .. controls (3.31,0.3) and (6.95,1.4) .. (10.93,3.29)   ;

%Curve Lines [id:da3486554685072205] 
\draw    (561,168) .. controls (503,120) and (460,134) .. (459,208) ;

%Straight Lines [id:da8669613088764891] 
\draw    (385,193) -- (265.94,223.5) ;
\draw [shift={(264,224)}, rotate = 345.63] [color={rgb, 255:red, 0; green, 0; blue, 0 }  ][line width=0.75]    (10.93,-3.29) .. controls (6.95,-1.4) and (3.31,-0.3) .. (0,0) .. controls (3.31,0.3) and (6.95,1.4) .. (10.93,3.29)   ;

% Text Node
\draw (322,189) node  [align=left] {$\displaystyle \nu $};
% Text Node
\draw (86,222) node  [align=left] {C'};
% Text Node
\draw (434,168) node  [align=left] {C};
% Text Node
\draw (240,24) node  [align=left] {$ $C};
% Text Node
\draw (97,33) node  [align=left] {C};
% Text Node
\draw (141,164) node  [align=left] {$\displaystyle \pi _{\eta }$};
% Text Node
\draw (193,97) node  [align=left] {};
% Text Node
\draw (314,291) node [scale=0.9] [align=left] {$\displaystyle  \begin{array}{{>{\displaystyle}l}}
	\ Figure\ 1:\\
	\nu :C\rightarrow C'\ normalization\ map,\ \pi _{\eta } :\ \tilde{C}\rightarrow C'\ admissable\ covering\ such\ that\ P( C',\eta ) =J( C) \ 
	\end{array}$};
% Text Node
\draw (84,126) node  [align=left] {$\displaystyle \tilde{C}$};
% Text Node
\draw (49,273) node  [align=left] {};
\end{tikzpicture}
\end{center}
% Conversely, for any $[P]\in Pr(\Delta_0^{''})$ there is by construction $[C]\in M_g$ such that $[P]=[JC]$. Summing up, we have \begin{equation}\label{eq:Jac}\overline{J_g}\simeq \overline{Pr}(\Delta^{''}_0)\subset \overline{P_{g+1}}\end{equation}
%and the fibre of the Prym map over $[JC]\in \overline{P_{g+1}}$ is at least of dimension $2$, obtained varying $x,y\in C$ as above.
%\smallskip
%\noindent 

\subsection{Admissible coverings, gonality and Clifford index}\label{sec:acgc}
Let $C$ be a smooth projective curve.  The gonality and the Clifford index of $C$ are defined as
\begin{align*}\label{def:goncliff} &\gon C= \min\{n\in \bN \,| \, C \mbox{ has a } g^1_n\} ;\\
&\Cliff C = \min\{\Cliff D=\deg D- 2 (h^0(D)-1)\, | \, D\subset C \mbox{ divisor}, h^0(D)\geq 2,  h^1(D)\geq 2  \}.
\end{align*}
%the gonality $\gon C$ of $C$ is the minimum integer $n$ such that $C$ carries a $g^1_n$, namely a degree $n$-map $C\to \bP^1$, and the Clifford index $\Cliff C$ is the minimum  $\Cliff D= \deg D- 2(h^0(C,D)-1)$ over the effective divisors $D \subset C$ such that $h^0(C, D)\geq 2$ and $h^1(C, D)\geq2$. 
 By \cite{CM91}, these are related by  \begin{equation}\label{eq:gon-cliff}
\gon C -3 \leq \Cliff C \leq \gon C- 2 ,
\end{equation}
where the second inequality is an equality on a general $C$ in $M_g$, while the first one is conjectured to be extremely rare \cite{ELMS89}. Furthermore, we will need the following well known facts 
\begin{enumerate}
	\item \label{en:goncliff}$\gon C \leq \left\lfloor \frac{g+3}{2}\right \rfloor$ (and so $\Cliff C \leq \left\lfloor \frac{g-1}{2}\right \rfloor$) and the equality holds for a general curve $[C]\in M_g$ (see \cite{ACGH});
	\item \label{en:goncliff2} let $f:\cC\to B$ be a fibration of projective curves over a complex curve $B$ and such that the general fibre is smooth. The invariants $\gon C_b $ and $\Cliff C_b$ are maximal on a general fibre $C_b$ of the set of smooth fibres (\cite{K99}).
\end{enumerate}
The point  $(\ref{en:goncliff2})$ does not say anything about the behaviour of these invariants on the singular fibres and their definition itself is subject of deeper studies. Anyway, we are only interested in estimating the gonality and the Clifford index of a smooth component $C$ in the normalization $\nu: C^{\nu}\to C'$ of a stable nodal curve $C'$, obtained by degeneration from a family of smooth curves, from those of the general smooth fibre of the family (according to $(\ref{en:goncliff2})$). For this, it is enough to use the theory of admissible coverings as developed by \cite{HM82} to compactify Hurwitz spaces, without introducing explicitly such invariants on singular curves. We follow \cite{HM82}  and also \cite{C85}. 

\begin{definition}\label{def:admissable covering} An {\em admissible covering} of degree $k$ with $m$ ramification points is the data of
	\begin{enumerate} \item a stable $m-$pointed reduced connected curve $(E, x_1, \dots ,x_m)$ of arithmetic genus $0$ (i.e. a curve with ordinary double points where any rational component is smooth and stable and the dual graph is connected);
		\item a reduced connected curve $X$ with ordinary double points and a morphism $\pi : X \to E$ of degree $k$ (everywhere) such that over any marked point $x_i$ of $E$, $X$ is smooth and $f$ has a unique simple ramification point $y_i$, on any smooth point of $E$, $f$ is \'etale and over double points $p$ of $E$, $X$ has an ordinary double point $q$ and $f$ is locally described as \begin{equation}
		X\,:\, xy=0; \quad E\,:\, uv=0;\quad \, f\,:\, u=x^{k'}\quad ,\, v= y^{k'}, 
		\end{equation}
		for some $k'\leq k$.
		\end{enumerate}
	\end{definition}
In particular, any simple covering over $\bP^1$ is an admissible covering and one can construct an admissible covering starting by any non symple covering over $\bP^1$. So they parametrize Hurwitz spaces with some fixed data (e.g. certain pencils $g^1_k$). Let $\Delta$ be the complex disk. The important property we need is that families of admissible coverings over $\Delta  \setminus \{0\}$ extend to families over $\Delta$. Indeed applying this, we have the following 
\begin{lemma}\label{lem:cliffnorm} Let $C'$ be a stable nodal curve of genus $g'$, let  $\nu': C^{\nu'}\to C'$ be its normalization and let $C\subset C^{\nu'}$ be a smooth connected curve of genus $g$.  Consider $f':\cC'\to \Delta$, a family of smooth projective curves of genus $g'$ over $\Delta \setminus \{0\}$, the complex disk minus zero, such that $C'=f^{-1}(0)$. Then any family  of pencils $g^1_k(t)$ over $ \Delta\setminus\{0\}$ compatible with $f$ (i.e. $g^1_k(t)$ is a pencil on $C'_t=f^{-1}(t)$) defines a pencil $g^1_{k'}$ on $C$ for some $k'\leq k$. In particular, $\gon C\leq \gon C'_t$ and $\Cliff C\leq \Cliff C'_t+1$, for a general fibre $C'_t=f^{-1}(t)$. 
	\end{lemma}
	\begin{proof} Let $C'$, $\nu':C^{\nu'}\to C'$, $C\subset C^{\nu'}$  and $f':\cC'\to \Delta$ as in the statement and denote by $C'_t$ the fibre of $t\in \Delta$. Consider a family of $g^1_k(t)$ over $\Delta\setminus \{0\}$ compatible with $f$, namely a family of coverings $\pi_t: C'_t\to \bP^1$ of degree $k$. By \cite{HM82} (see also \cite{C85}), the family extends over $0$ by taking an admissible covering $\pi: X\to E$ such that the stable model of $X$ is $\St(X)=C'$ (according to notations of Definition \ref{def:admissable covering}). By construction $X$ contains $C$, $\pi$ has degree $k$ over the smooth points of $E$. Since irreducible components maps onto irreducible components, $\pi$ defines a $g^1_{k'}$ on $C$ for some $k'\leq k$. To obtain the bound on the gonality recall that by (\ref{en:goncliff}) and (\ref{en:goncliff2}) the gonality $\gon C'_t$ is maximal over an open dense subset $\Delta'$ of $\Delta\setminus \{0\}$ and this constructs up to a finite base change a family over $\Delta'$ of coverings $\pi_t:C'_t\to \bP^1$ of degree $k=\gon C'_t$. By the argument above these define a $g^1_{k'}$ on $C$ for some $k'\leq k$ and so $\gon C\leq \gon C'_t$. 
			Applying \eqref{eq:gon-cliff}, we immediately conclude 
		$	\Cliff C \leq \gon C-2\leq \gon C'_t-2\leq \Cliff C'_t+3-2=\Cliff C'_t+1$.
		
		\end{proof}

\subsection{Variations of Hodge structures and geodesics in the moduli space of principally polarized abelian varieties}\label{sec:vhs}
\subsubsection{Weight $1$ polarized variations of Hodge structures }

Let $Y$ be a smooth complex variety and let
	$(\vhs, \cH^{1,0}, \cQ) $ be a polarized variation of Hodge  structures (pvhs, in short) of weight 1 on $Y$. Namely, $\bH_Z$ is a local system of lattices, $\cH^{1,0}$ is a Hodge bundle of type $(1,0)$ (equivalently, the Hodge filtration in this case) and $\cQ$ is a polarization. Let $\bH_{\bC}=\bH_\bZ\otimes_{\bZ}\bC$  and let $\cH=\bH_\bC \otimes \cO_Y$ be the holomorphic flat bundle with the holomorphic flat connection $\nabla$ defined by $\ker \nabla \simeq \bH_\bC$, the Gauss-Manin connection. The holomorphic inclusion $\cH^{1,0}\subset \cH$ of  vector bundles induces the short exact sequence
	\begin{equation}
	\label{ses-vhs}
	\begin{tikzcd}
	0 \arrow{r} & \cH^{1,0} \arrow{r} & \cH \arrow{r}{\pi^{{0,1}}} &
	\cH/\cH^{1,0} \arrow{r}& 0.
	\end{tikzcd}
	\end{equation}
	Let $\pi^{{0,1}'}: \cH\otimes \Omega^1_Y\to \cH/\cH^{1,0}\otimes \Omega^1_Y$ be the map induced by $\pi^{0,1}$ and $\sigma =\pi^{{0,1}'}\circ \nabla: \cH^{1,0}\to \cH/\cH^{1,0}\otimes \Omega_Y^1$ the second fundamental form of $\cH^{1,0}\subset \cH$ with respect to $\nabla$. 
	
%	The Hodge metric, fibre wise defined by
%	$h(v, w) := i \cQ(v, \bar{w})$, is positive definite on
%	$\cH^{1,0}$. The associated orthogonal projection $p: \cH\to \cH^{1,0}$ splits $\cH= \cH^{1,0}\oplus {\cH^{1,0}}^{\perp}$ at the level of complex vector bundles (namely, ${\cH^{1,0}}^{\perp}\simeq \cH/\cH^{1,0}$ as a complex vector bundles but not as holomorphic  bundles if the variation is not trivial), and defines a connection
%	$\nh=p\nabla_{|\cH^{1,0}}: \cH^{1,0}\to \cH^{1,0}\otimes \cA^1_Y$, which is the Chern connection of the Hermitian bundle $(\cH^{1,0}, h)$. By definition, $\ker \nabla_{|\cH^{1,0}} =\ker \nabla^{hdg}$.
%	We are interested in attaching two nested vector subbundles to any weight $1$-polarized variation of Hodge structures. 
Following \cite{GPT18}, let $\bU=\ker \nabla_{|\cH^{1,0}}$ and define 
		\begin{gather}\label{eq-UK}
			\cU:=\bU\otimes
			\cO_Y,\quad \quad
			\cK:=\ker ( \sigma : \cH^{1,0} \lra \cH /\cH^{1,0} \otimes
			\Omega_Y^1).
		\end{gather}

Then $\cU$ is a holomorphic vector bundle and $\cK$ is a coherent sheaf which is a vector bundle when $\sigma$ is of constant rank.		
% Indeed,   $\cU$ is a holomorphic unitary flat bundle, simply because $\bU=\ker \nabla_{|\cH^{1,0}}= \ker \nabla^{hdg}$. The kernel $\cK$ in principle is only a coherent subsheaf over $Y$ since $\nabla$ is flat and then $\sigma$ is holomorphic and $\cK\subset \cH^{1,0}$. If $Y=B$ is a curve, then $\cK$ is a vector bundle. Indeed, it is a subsheaf $\cK\subset \cH^{1,0}$ of a vector bundle over a curve, so a vector bundle itself. It is furthermore  a vector subbundle $\cK\subset \cH^{1,0}$. Indeed, the quotient $\cO_{\cK}$ with respect to the inclusion $\cK\subset \cH^{1,0}$ has a natural inclusion $\cO_{\cK}=\cH^{1,0}/\cK\subset \cH/\cH^{1,0}\otimes \Omega^1_B $ into a vector bundle and so it is itself a vector bundle. Summing up, 
%\begin{equation}
%\label{ses-vhs4}
%\begin{tikzcd}
%0 \arrow{r} & \cK \arrow{r}& \cH^{1,0} \arrow{r}{\sigma} & 
%\cQ_{\cK} \arrow{r}& 0, \,\, &
%\end{tikzcd}
%\end{equation}
%is a s.e.s involving only vector bundles and so $\sigma$ must have constant rank.
\begin{definition} 
	\label{defUK}
%	Let $\bU=\ker \nabla_{|\cH^{1,0}}$. We define \begin{itemize}
%		\item $\cU:=\bU\otimes
%		\cO_Y,$ % with $\bU=\ker \nabla_{|\cH^{1,0}}$;
%		\item
%		$\cK:=\ker ( \sigma : \cH^{1,0} \lra \cH /\cH^{1,0} \otimes
%		\Omega_Y^1)$.
%	\end{itemize}
%	and 
	We call $\cU$ and $\cK$ as defined in \ref{eq-UK}, {\em the unitary flat bundle} and {\em the
		kernel sheaf} of the variation, respectively.
\end{definition}

In the case $Y=B$, where $B$ is a complex smooth curve, $\cK$ is a vector sub-bundle of $\cH^{1,0}$. We are interested to compare these two bundles $\cU$ and $\cK$ on a smooth complex curve $B$.
For this, consider the exact sequence
\begin{equation*}
\begin{tikzcd}
0 \arrow{r} & \cK \arrow{r} & \cH %:=R^1\pi_* \bC\otimes \cO_B
\arrow{r}{\pi} & \cH/\cK \arrow{r} & 0.
\end{tikzcd}
\end{equation*}
Let $\pi': \cH\otimes \Omega^1_Y\to\cH/\cK\otimes \Omega^1_{Y}$ be the map induced by $\pi$ and let $\tau=\pi'\circ \nabla_{|\cK}: \cK\to \cH/\cK\otimes \Omega^1_{B}$ be the second fundamental form of $\cK\subset \cH$.
\begin{proposition}\label{prop-UK} We
	have $ \cU\subset \cK$ and if $\tau\equiv 0$, then $\cU=\cK$.
\end{proposition}
See \cite[Proposition 2.2]{GPT18} for the proof.

\smallskip
\noindent
We want to study certain p.v.h.s. of weight $1$ on smooth complex (local) curves $B\subset A_g$ that we recall below. For any $B\subset A_g$ smooth complex curve, let $f: \cA\to B$ be the family of abelian varieties (defined up to finite base change) and let $A_b$ be the fibre over $b\in B$. Then the p.v.h.s. is defined by
\begin{eqnarray}\label{eq:vhsAV}
\bH_{\bZ}\simeq R^1f_\ast \bZ, &\quad \quad \cH^{1,0}= f_\ast \Omega^1_{\cA/B}\subset \cH= R^1f_\ast \bC\otimes \cO_{B};& 
\end{eqnarray}
 \begin{eqnarray*} \mbox{where }\quad (R^1f_\ast \bZ)_b\simeq H^1(A_b, \bZ), & (f_\ast\Omega^1_{\cA/B })_b\simeq H^0(A_b, \Omega^1_{A_b}), &{(R^1f_\ast \bC\otimes \cO_{B})}_b\simeq H^1(A_b, \bC). 
 \end{eqnarray*}
%Recall that the moduli space of abelian varieties $A_g$ is the parametrizing space of weight $1$ variations of Hodge structures (see \cite{G_ATrascendental_1971}). 
In particular, we are interested in this cases: 
\begin{enumerate}
	\item for $B'\subset J_g$, let $A_b=JC_b$ for $[C_b]\in M_g$ , the p.v.h.v. \eqref{eq:vhsAV} is given by
\begin{eqnarray}\label{eq:jacvhs}
	\bH_{\bZ,b}\simeq H^1(C_b,\bZ), & &\cH^{1,0}_b\simeq H^0(C, \omega_{C_b});
	\end{eqnarray} 
	Furthermore, let $f': \cC\to B$ be the associated family of curves. Then \eqref{eq:jacvhs} corresponds exactly fibre-wise to the geometric p.v.h.s. defined by the family of curves as
	\begin{eqnarray}\label{eq:vhsC}
	\bH_{\bZ}\simeq R^1f'_\ast \bZ,  & \cH^{1,0}= f'_\ast\Omega^1_{\cC/B}\subset \cH= R^1f'_\ast \bC\otimes \cO_{B}; &
	\end{eqnarray}
	\item for $B'\subset P_{g+1}$, let $A_b=P(C'_b, \eta_b)=P(\pi_b)$, where $[C'_b, \eta_b]\in R_{g+1}$ and $\pi_b:\tilde{C}_b\to C'_b$ is the associated \'etale double covering. The involution $\sigma_b:\tilde{C}_b\to \tilde{C_b}$ induced by $\eta_b$ splits \begin{align}\label{eq:splitPrym}
	&H^1(\tilde{C}_b, \bZ)\simeq H^1(\tilde{C}_b, \bZ)^+\oplus H^1(\tilde{C}_b, \bZ)^-
	&H^0(\tilde{C}_b, \omega_{\tilde{C}_b})\simeq H^0(\tilde{C}_b, \omega_{\tilde{C}_b})^+\oplus H^0(\tilde{C}_b, \omega_{\tilde{C}_b})^-
	\end{align}
	into invariant (denoted by $+$) and antiinvariant (denoted by $-$) parts such that 
		\begin{eqnarray*}
		H^1(\tilde{C}_b, \bZ)^+\subset  H^0(\tilde{C}_b, \omega_{\tilde{C}_b})^+ ,\quad &  H^1(\tilde{C}_b, \bZ)^-  \subset H^0(\tilde{C}_b, \omega_{\tilde{C}_b})^-
		\end{eqnarray*}
		 are lattices. Furthermore, 
	\begin{eqnarray*}
	H^0(\tilde{C}_b,  \omega_{\tilde{C}_b})^+\simeq H^0(C'_b, \omega_{C'_b}), \quad  & H^0(\tilde{C}_b,  \omega_{\tilde{C}_b})^-\simeq  H^0(C'_b,\omega_{C'_b}\otimes \eta_b)
\end{eqnarray*}
    So the p.v.h.v. \eqref{eq:vhsAV} is given by 
	\begin{eqnarray}\label{eq:prymvhs}
	\bH_{\bZ,b}^-\simeq H^1(\tilde{C}_b, \bZ)^-, & &{\cH^{1,0}}^-_b\simeq H^0(\tilde{C}_b, \omega_{\tilde{C}_b})^-\simeq H^0(C'_b,\omega_{C'_b}\otimes \eta_b).
	\end{eqnarray}
	Furthermore, let $\pi: \tilde{\cC}\to \cC'$ be the family of coverings over $B$ and let $\tilde{f}: \tilde{\cC}\to B$ and $f': \cC'\to B$ be the associated families of curves. They have their own p.v.h.s. defined by \eqref{eq:vhsC} of case $(1)$, satisfying fibre-wise \eqref{eq:splitPrym}.
	\end{enumerate}
	In both cases the polarizations are induced by the intersection form in a natural way.

	We are interested in studying local curves $B'\subset A_g$ that come from local curves $B$ in the compactifications $\overline{M_g}$ and $\overline{M_{g'}}$ of the moduli spaces of curves of genus $g$ and $g'=g+1, g\geq2$ respectively, with the property that these intersect the boundary in at most an isolated point $b_0\in B$. 
\begin{remark}\label{rem:extU}In the case that $B$ intersects the boundary in $b_0$ both $\cU$ and $\cK$ are a priori defined only on $B^0=B\setminus \{b_0\}$, according to Definition \ref{defUK}. But these naturally extend to the whole $B$ by using the Deligne extension (see e.g. \cite{G_ATrascendental_1971}) plus the unipotency monodromy theorem. Furthermore,  the extension is trivial on $\cU$ (see e.g. \cite{PT} and also \cite{CD:Answer_2017} for details). \end{remark}
%%
%%We recall that a curve $B\subset M_g$ corresponds (up to finite base change) to a semistable fibration $f:\cX\to B$ of projective curves such that the general fibre is smooth and the singular fibre has nodes as singularities. Let $f^0:\cX^0\to B^0$ be the restriction to the locus of smooth fibres and let $X_b$ be a smooth fibre of $f$. On the smooth locus $B^0$ there is the associated geometric variation of Hodge structures of weight $1$ is $(\bH_{\bZ}\simeq H^1(X_b, \bZ), \cH^{1,0}\simeq f^0_\ast \Omega^1_{\cX^0/B^0}, Q^0)$ together with the unitary flat bundle $\cU$ and the kernel bundle $\cK$ as defined above. The extended Hodge bundle 
%%is in this case $\widehat{\cH^{1,0}}=f_\ast\omega_{\cX/B}$ (see e.g. \cite{PetSteen_Mixed_2008}). 
%%		
%%%		\begin{corollary} Let $f: \cX\to B$ be a semistable fibration of projective varieties over a smooth complex curve such that the general fibre is smooth. Then the Hodge  bundle $f_\ast \Omega^1_{\cX/B}$ has a unitary flat subbundle $\cU\subset  f_\ast \Omega^1_{\cX/B}(\log f^{-1}(B_0))$ as in Proposition \ref{prop-ufb} defined by the canonical Deligne extension.
%%%			\end{corollary}
%%
	Notice that in the case of $B$ quasi-projective and $\cX$ a surface this is nothing but the Fujita decomposition \cite{Fuj78b}.
\subsubsection{Geodesics in the moduli space of principally polarized abelian varieties}
Let $H_g$ denote the Siegel upper half space. As a symmetric space of non-compact type it is endowed by a symmetric metric $h^s$, called the Siegel metric, defining a metric connection $\nabla^{LC}$ on the tangent bundle $TH_g$. As a parametrizing space of weight $1$ p.v.h.s., it carries a universal p.v.h.s. $(\bH_{\bZ},\cH^{1,0}, \cQ)$ with its Hodge  metric defined by $Q$, inducing a metric $h$ together with a metric connection $\nabla^{hdg}$ on $\Hom (\cH^{1,0}, \cH/\cH^{1,0})$. There is a natural inclusion 
\begin{eqnarray}\label{eq:SieHdg} 
(TH_g, \nabla^{LC})\subset (\Hom (\cH^{1,0}, \cH/\cH^{1,0}), \nabla^{hdg})
\end{eqnarray} compatible with the metric structure (see e.g. a classical reference  \cite{G_ATrascendental_1971} or some more recent references \cite{G18}, \cite{GPT18}).

\smallskip
We are interested in studying {\em (local) geodesics of $A_g$}. For this, we consider the universal covering $\psi: H_g\to A_g$  and the metric properties introduced before on $H_g$. Let $[A]\in A_g$ and let $\zeta\in T_{[A]}A_g$. 
%\blue{How is $T_[A]A_g$ if $[A]$ is a singular point?}. 
Take $\tilde{A}\in \psi^{-1}([A])$ and consider $\zeta$ as $\zeta \in T_{\tilde{A}}H_g\simeq T_{[A]}A_g$. Then in $H_g$ a (local) geodesic at $(\tilde{A}, \zeta)$ is simply a curve $\gamma:(-\epsilon, \epsilon)\to H_g$ such that $\gamma(0)=\tilde{A}$ and $\gamma'(0)=\zeta$  satisfying  $\nabla^{LC}_{\gamma'}\gamma'=0$. Working locally, we can assume w.l.o.g that $\gamma((-\epsilon, \epsilon))$ is contained in one sheet of $\psi$.

\begin{definition}\label{def:geoA} Let $[A]\in A_g$ and let $\zeta\in T_{[A]}A_g$ 
%	\blue{How is $T_[A]A_g$ if $[A]$ is a singular point?}. 
A (local) geodesic associated to $([A], \zeta)$ is a map $\psi\circ \gamma: (-\epsilon, \epsilon)\to A_g$, where $\gamma$ is a local geodesic in $H_g$ defined as above. 
	\end{definition}

We are interested in points of $J_g\subset A_g$ and directions in $T_{[A]}J_g\subset T_{[A]}A_g$. If  $[A]=[JC]$, namely the Jacobian of some $[C]\in M_g$, and $\zeta = \cup \xi$, for some $\xi \in H^1(C, T_C)$, under the isomorphisms $T_{[C]}M_g\simeq H^1(C, T_C)$ and $T_{[J_g]}A_g\simeq \Sym^2H^1(C, \cO_C)\simeq \Sym^2H^0(C, \omega_C)^\vee$, we will also refer to the geodesic at $(JC, \zeta)$ as the geodesic at $(C, \xi)$. This is admitted since the Torelli map is an immersion outside the Hyperelliptic locus and we are not considering hyperelliptic curves.  

A possible approach to the study of geodesics in $A_g$ involves variation of Hodge structures of weight $1$ by relating the Siegel metric with the Hodge metric (and their connections) under the identification above. 
We have the following (see \cite[Lemma 3.3 and Lemma 3.4]{GPT18} for the proof)
\begin{lemma}\label{prop:ciccia} Let $\gamma:(-\epsilon, \epsilon)\to A_g$ be a local geodesic associated to $([A], \zeta)$. Then there exists a complex curve $B\subset H_g$ containing the geodesic and such that $\cK =\cU$.
	\end{lemma}
	Notice that we can shrink $B$ around $\gamma((-\epsilon, \epsilon))$ in such a way that it is contained in one sheet on $\psi$ and so we can then look at it as a curve in $A_g$. 
%	\begin{remark}  By \cite[Theorem 1.3]{LZ_TheOort_2014}, a smooth close curve in $A_g$ is totally geodesic if and only if $\cK=\cU$ and so the above is a generalization of one implication in the local case. 
%		\end{remark}

\subsection{Useful isotropic spaces on unitary flat bundles}\label{sec:isotropic}
\begin{definition}\label{def-isotropic}
%	Let $V, T$ be two vector spaces and let $\alpha: \bigwedge^2V\to T$ be a linear map. The subspace $W\subset V $ is called {\em isoptropic } with respect to $\alpha$ if $\alpha_{|\bigwedge^2 W}\equiv0$.
	
	Let $C$ be a smooth projective curve, let $\cF$ be a rank $2$ vector bundle over $C$ and  let $\alpha: \bigwedge^2H^0(C,\cF)\to H^0(C,\det \cF)$ be a linear map. A subspace $W\subset H^0(C,\cF) $ is called {\em isoptropic } with respect to $\alpha$ if $\alpha_{|\bigwedge^2 W}\equiv0$.

	Let $\cL, 	\cL'$ be two line bundles on $C$, let $0\to \cL \to \cF \to \cL'\to 0$ be a s.e.s. associated to $\xi\in \Ext^1_{\cO_C}(\cL', \cL)$ and let $\alpha: \bigwedge^2H^0(C,\cF)\to H^0(C,\det \cF)\simeq H^0(C,\cL\otimes \cL')$. A subspace $V\subset H^0(C,\cL')$ is called {\em isotropic} with respect to $\alpha$ if it lifts to a subspace $W\subset H^0(C,\cF)$ isotropic with respect to $\alpha$.
	 \end{definition}
	 Notice that a subspace $V \subset H^0(C,\cL')$ lifts to $W\subset H^0(C,\cF)$ if and only if it lies in the kernel of the coboundary morphism $\delta: H^0(C,\cL')\to H^1(C,\cL)$ on the long exact sequence in cohomology.

	 	We are interested in the case of isotropic subspaces defined by the s.e.s. 
	 	\begin{equation}
	 	\label{ses:firstext}
	 	\begin{tikzcd}
	 	0 \arrow{r} & \cO_C \arrow{r} & \cF
	 	\arrow{r}& \om _C \arrow{r} & 0 &: \xi
	 	\end{tikzcd}
	 	\end{equation}
	 	 and with respect to the associated map $\alpha: \bigwedge^2H^0(C,\cF)\to H^0(C,\det \cF)\simeq H^0(C,\omega_C)$. Namely, by definition $V\subset H^0(C, \omega_C)$ is isotropic with respect to $\alpha$ if there exists a filting $\tilde{V}\subset H^0(C, \cF)$ such that $\alpha(\bigwedge^2V)=0$. 
\begin{theorem}\label{thm:cdfU} Let $f: \cC \to B$ be a fibration of smooth projective curves over a smooth complex curve $B$ and let $\cU\subset f_\ast \omega_{\cC/B}$ be the associated unitary flat bundle (Definition \ref{defUK} ). Assume that the fibres $U_b\subset H^0(\omega_{C_b})$ are isotropic subspaces with respect to $\alpha_b$ for any $b\in B$. Then (up to a finite base change) there exists a smooth projective curve $\Sigma$ of genus $g'=\rk \cU$ and a non constant fibre-preserving map $\varphi: \cC \to \Sigma$ such that $U_b\simeq \varphi^\ast H^0(\Sigma,\omega_{\Sigma})$.
	\end{theorem}
	
	In this case, isotropic subspaces of $H^0(C,\omega_C)$ are Massey-trivial as in \cite[Definition 2.2]{PT} and, under this arrangement, we refer to \cite[section 2]{GPT18}) or directly \cite[Theorem 4.4]{PT} for the proof.  Just to give an idea, the proof  can be recovered putting together \cite[Lemma 3.2 and Proposition 4.3]{PT}  and \cite[Theorem 1.5]{GST}. The sketch of the proof is the following: let $\bU\subset \cU$ be the local system underlying $\cU$ and for any $b\in B$ 
	\begin{itemize}
		\item[(i)] by \cite[Lemma 3.2]{PT}, there exists a lifting $l:\bU\simeq  \tilde{\bU}\subset \Omega^1_{S,d},$ where $\Omega^1_{S,d}=\ker \{d:\Omega^1\to \Omega^2\}$;
		\item[(ii)] by \cite[Proposition 4.3]{PT} if $U_b$ is isotropic with respect to $\alpha_b$ for any $b\in B$, then there is a basis $s_1, \dots, s_r$ of local sections of $\bU$ which is pointwise a base of $U_b$, under the isomorphisms $U_b\simeq \bU$ (equivalently a local base of flat sections of $\bU$) such that $\omega_1=l(s_1), \dots \omega_r=l(s_r)$ is a set closed holomorphic forms on $\cC$ such that $\omega_i\wedge \omega_j=0,$ for any$i\neq j.$
		\item[(iii)] by \cite[Theorem 1.5]{GST}, applied to the setting constructed in $(ii)$, we obtain a fibre-preserving map $\phi: \cC\to \Sigma$ over a curve $\Sigma$ of genus $g(\Sigma)\geq 2$ such that for any $i$, $\omega_i\in \phi^\ast H^0(\Sigma, \omega_{\Sigma})$ and so by  construction $U_b\simeq \phi^\ast H^0(\Sigma, \omega_{\Sigma}).$ 
		\end{itemize}
%		 We just sketch that the proof uses deep results: the lifting Lemma  \cite[Lemma 3.2]{PT} of flat sections of $\cU$ to "tubular" closed holomorphic forms on $\cC$ and the "tubular" Castelnuovo - de Franchis Theorem \cite{GST} (see also \cite{GPT18} for a shorter explanation).
%	
	 To produce useful isotropic subspaces we will use the following (see \cite[Lemma X.7]{Bea83})
	 \begin{lemma}\label{lem-decel}Let $C$ be a smooth projective curve and let $\cF$ be a vector bundle of rank $2$ on $C$. If 	\begin{eqnarray}2h^0(C,\cF)-3 > h^0(C,\det \cF ),\end{eqnarray}
	 	then the kernel of the map $\alpha$ has a decomposable element defining (up to saturation over the base locus) a line bundle $\cL \subset \cF$ such that $h^0(C,\cL) \geq  2$ and $\cF/\cL$ is a line bundle.
	 \end{lemma}
	 Examples are constructed by taking $\cF$ defined by \eqref{ses:firstext} and $\cL$ such that $h^0(C,\cF/\cL)=1$ (as we will see below).

\section{The theorem in the Jacobian locus}\label{sec:Jacobian}
%For $g\geq 7$, let $[C]\in M_g$ and $\xi \in H^1(C, T_C)\simeq T_{[C]}M_g$. Assume that the rank $\rk \xi= k<\Cliff C-2$, i.e. that the associated cup-product map  $\cup\xi:H^0(C,\omega_C)\to H^0(C,\omega_C)^\vee\in \Sym^2 H^0(C, \omega_C)^\vee\simeq \Sym^2 H^1(C, \cO_C)$ has rank $k$ and then $\dim \ker \cup \xi =g-k$. 
%
%\smallskip
%\noindent
	In the following, we answer negatively to the preliminary problem if geodesics defined as in Theorem \ref{Thm-Main1} stay inside the Jacobian locus $\overline{J_g}$. The proof in this case is a straightforward application  \cite[Theorem 1.1]{GPT18}.
 
 \begin{lemma}\label{lem:Jacobians}For $g\geq 7$, let $[C]\in M_g$ and let $\xi \in H^1(C, T_C)$ such that $\rk \xi = k<\Cliff C$. Consider the geodesic $\gamma$ associated to $(C, \xi)$. Then for a sufficiently small $\epsilon \in \bR^+,$ $\gamma((-\epsilon, \epsilon))\cap J_g=[JC]$. 
 	\end{lemma}
 \begin{proof} Let  $\gamma$ be the geodesic associated to $(C, \xi)$ (Definition \ref{def:geoA}) and satisfying the assumptions of the lemma. 
 	Let $t=\gamma(s)$ and consider the s.e.s. associated to $\xi_t\in H^1(C_t,T_{C_t})$
 	\begin{equation}
 	\label{ses:firstextensionC}
 	\begin{tikzcd}
 	0 \arrow{r} & \cO_{C_t} \arrow{r} & E_t
 	\arrow{r}& \om _{C_t} \arrow{r} & 0 &: \xi_t .
 	\end{tikzcd}
 	\end{equation}
 	The coboundary morphism on the long exact sequence induced in cohomology is computed by the cup product with the extension class, $\cup \xi_t : H^0(C_t,\omega_{C_t})\to H^1(C_t, \cO_{C_t})$. Consequently, it corresponds to the tangent vector $\gamma'(s)\in T_{[JC_t]}A_g$ at $t=\gamma(s)$. Since $\gamma$ is a geodesic, we have that any $t=\gamma(s),$ $\dim \ker \cup\xi_{t}=g-k$  Indeed, by assumption for $t=\gamma(0)$,   $\dim \ker \cdot \cup \xi=g-\rk \cdot \cup \xi =g-k$ and moving by parallel transport, for $t=\gamma(s)$, we obtain $\dim \ker \cup\xi_{t}=\dim \ker \cup\xi_{\gamma(s)}=\dim \ker \cup \xi_{\gamma(0)}=\dim \ker \cup \xi$. 
 	
 	\vspace{1.5mm}
 Consider a map $\alpha_t:\bigwedge^2 H^0(C_t, E_t)\to H^0(\det E_t)\simeq H^0(C_t,\omega_{C_t})$. We want to prove that $\ker \cup \xi_t$ is isotropic with respect to $\alpha_t$, namely that $\alpha_t(\bigwedge^2 \ker \cup\xi_t) =0$.  To do this we, use Lemma \ref{lem-decel} to to construct a useful line bundle starting by a decomposable element. Observe that the assumption $2 h^0(C'_t,E_b)-3 > h^0(C_t, \omega_{C_t})$ of Lemma \ref{lem-decel} is satisfied since $h^0(C_t, E_t)= h^0(C_t,\cO_{C_t})+ \dim \ker \cup \xi_t=1+g-k$ and $k<\Cliff C\leq (g-1)/2$ (see  (\ref{en:goncliff}), Subsection \ref{sec:acgc}). Consequently,  the kernel of the map $\alpha_t$ has a decomposable element and such an element defines (up to saturation) a line bundle $ \cM_t$ such that $h^0(C, \cM_t)\geq 2$ and a diagram 
 	
 	\begin{equation}\label{dia:cliffC}
 	\begin{tikzcd}
 	& & 0 \arrow{d} & & \\
 	& & \cM_t \arrow{d}{i_2} & & \\
 	0 \arrow {r} & \cO_{C_t} \arrow{r}{i} \arrow{rd}{s'} & E
 	\arrow{d}{p_{2}} \arrow{r} &
 	\om_{C_t} \arrow{r} & 0 \\
 	& & \omega_{C_t}\otimes \cM^{-1}_t\arrow{d} & &\\
 	&& 0 &&
 	\end{tikzcd}
 	\end{equation}
Observe that $H^0(C_t, \cM_t)\subset H^0(C_t, E_t)$ is by definition isotropic with respect to $\alpha_t$  and furthermore $H^0(C_t, \cM_t)\subset \ker \cup \xi_t$. To prove that $\ker \cup \xi_t$ is isotropic we show that $\ker \cup \xi_t=H^0(C_t, \cM_t)$, if $\Cliff(C)>k$. Equivalently, that $h^0(C_t, \omega_{C_t}\otimes \cM^{-1}_t)=1$, if $\Cliff(C)>k$.  First of all $h^0(C_t, \omega_{C_t}\otimes \cM^{-1}_t)\geq 1$, since $s'\neq 0$ (otherwise $\cL\simeq \cO_C$), which is not admitted). By contradiction, let us assume that $h^0(C_t, \omega_{C_t}\otimes \cM^{-1}_t)\geq 2$. In this case both $h^0(C_t,\cM_t)$ and  $h^0(C_t, \omega_{C_t}\otimes \cM^{-1}_t)$ have at least two sections and so by definition  $\cM_t$ contributes to $\Cliff(C_t)$ and $\Cliff C_t \leq \Cliff{\cM_t}$ (see (\ref{en:goncliff2}), Subsection \ref{sec:acgc}). By (\ref{en:goncliff2}) of Section \ref{sec:acgc}, for a general $t$, $\Cliff(C)\leq\Cliff (C_t)\leq \Cliff(\cM_t)$. Comparing  $h^0(C_t, \omega_{C_t}\otimes \cM^{-1}_t)+h^0(C,\cM_t)\geq g-k+1$ (given by the diagram) and $h^0(C_t, \cM_t)-h^0(C_t, \omega_{C_t}\otimes \cM^{-1}_t)=\deg \cM_t - g+1$ (given by the Riemann-Roch formula), one has $2h^0(C_t, \cM_t)\geq \deg \cM_t-g+1+g-k+1$ and so $\Cliff (C)\leq\Cliff (\cM_t)\leq k$, contradicting our assumptions.
 
% \blue{sist cliff}
 
 \smallskip
 \noindent
 The argument above proved that $\ker \cup\xi_t= h^0(C_t, \cM_t)$ and so that it is isotropic with respect to $\alpha_t$,  for $t$ general on $\gamma$. We now prove that this kind of geodesic does not lie in the Jacobian locus. First of all, since $\gamma((-\epsilon, \epsilon))\subset J_{g}$ is real analytic, by Lemma \ref{prop:ciccia} we can take a local holomorphic extension to a smooth complex curve $B\supset \gamma((-\epsilon, \epsilon))$ and let $f:\cC\to B$ be the corresponding family of curves. Namely,  $C_t=f^{-1}(t)$ maps to $t=[JC_t]$. Let $U_t$ be the fibre over $t$ of the unitary flat bundle $\cU$ of the family $f:\cC\to B$ (Definition \ref{defUK} and Remark \ref{rem:extU}, case $(1)$). Since $\gamma((-\epsilon, \epsilon))\subset B$ is a geodesic, $\ker \cup \xi_t  = U_t$ for a general $t$ and so $U_t\subset H^0(\omega_{C_t})$ is isotropic with respect to $\alpha_t$ for the general $t\in B$ (we can repeat the previous argument using the whole $B$). Then, by Theorem \ref{thm:cdfU},
% By the Castelnuovo de-Franchis theorem for tubular surfaces and the lifting Lemma (see \cite{PT}, \cite{GST}), 
 we have a non constant morphism ${C_t}\to \Sigma$, for any $t\in B$, where $\Sigma$ is a curve of genus $g(\Sigma)$ equal to the rank of $\cU$ and by the Riemann-Hurwitz formula we have $g(\Sigma)=\rk \cU\leq (g+1)/2$ since $f$ is non isotrivial. 
 
 To conclude, we show that this is not possible. By Lemma \ref{prop:ciccia}, $\rk \cU=\rk \cK$,  and so $\rk \cU=\rk \cK = \dim \ker \cdot \cup \xi=g-k$. Since $k<\Cliff C\leq (g-1)/2$ (by assumption and (\ref{en:goncliff2}) of Subsection \ref{sec:acgc}), we obtain $g-k>(g+1)/2$, which is not admitted.  \end{proof}
  
\section{Proof of the main result}\label{sec:draft}
      
%  \begin{lemma}\label{lem:thm1}Let $\xi \in H^1(C, T_C)$ such that $\rk \xi = k<\Cliff C-2$. Let $\gamma$ be the geodesic associated to $(C, \xi)$. Then $\gamma((-\epsilon, \epsilon))\cap \overline{P_{g+1}}\setminus Z=[JC]$. 
%  \end{lemma}
%\begin{proof}[Proof of Theorem \ref{Thm-Main1}] 
In this section we prove Theorem \ref{Thm-Main1}. Let $JC$ be a general Jacobian variety, let $\zeta \in T_{[JC]}J_g$ such that $\rk \zeta= k<\Cliff C-3$ and let $\gamma$ be the geodesic associated to $(C, \zeta)$ (Definition \ref{def:geoA}).  We consider $\overline{J_g}\subset  P(\Delta_0^{''})\subset P(S)\subset \overline{P_{g+1}}$ (see \ref{eq:Jac}, Subsection \ref{sec:JacPrym}). First of all, we observe that by assumption $\Cliff C >2$ and so by Corollary \ref{cor:genJac}, through all the proof, if $P$ is a generalized Prym variety such that $P=JC$, then $P$ is of type $(\ast\ast)$ and not $(\ast)$ (Definition \ref{def:genPrym}, Remark \ref{rem-type*}) and furthermore provided by Lemma \ref{lem:normJacPrym}. Secondly, taking $JC$ general here means that 
\begin{itemize}
%	\item[(i)] it is not realized as a generalized Prym variety with $n_e=c_e=0$ (equivalently, of type $(\ast)$), according to Corollary \ref{cor:genJac};
	\item[(i)] it is not the Jacobian of a curve $C$ lying in a specific locus $W_{(P)}\subset M_g$ that will be precisely defined in case $3$ below.
	\end{itemize}
	 In this setting, we prove that $\gamma$ is not contained in $\overline{P_{g+1}}$.
%By definition this means that the intersection of $\gamma$ with the interior of $\overline{P_{g+1}}$ is not open Zariski and so we have to prove that we have that up to shrink $(-\epsilon, \epsilon)$ if necessarily, $(\gamma((-\epsilon, \epsilon))\setminus {[JC]})\cap P_{g+1}=\emptyset$. By contradiction assume that there is $\gamma$ contained in $\overline{P_{g+1}}$, then $\gamma ((-\epsilon, +\epsilon))\setminus {[JC]}\subset P_{g+1}$
%Let $\xi \in H^1(C, T_C)$ such that $\zeta= \cup \xi\in \Sym^2H^1(C, \cO_C)$ via the Torelli map. 
%\blue{check: there are no families of generalized Prym varieties which are not Jacobians and converging to a Jacobian. If it is true, then}
	 By Lemma \ref{lem:Jacobians}, we have already proven that $\gamma$ is not contained in the Jacobian locus $\overline{J_g}$ and so up to shrink $(-\epsilon, +\epsilon)$ we can assume $\gamma((-\epsilon, +\epsilon))\setminus\{[JC]\}\subset P_{g+1}$, which means that $\gamma$ is  parametrizing a family of standard Prym varieties. 
	 Furthermore, since $\gamma((-\epsilon, \epsilon))\subset A_g$ is real analytic, we can take a local holomorphic extension to a smooth complex curve $B\supseteq\gamma((-\epsilon, \epsilon))$ as constructed by Lemma \ref{prop:ciccia}. In particular, $B$ has the following properties: $B\subset \overline{P_{g+1}}$ by the identity principle and $B\setminus \{[JC]\}\subset P_{g+1}$ because $\gamma((\epsilon, \epsilon))\cap J_g=[JC]$, $\gamma(-\epsilon, \epsilon)\setminus\{[JC]\}\subset P_{g+1}$ and $J_g\cap \overline{P_{g+1}}$ must be complex analytic. Let $B^0=B\setminus \{[JC]\}$. Then on $B^0$ we have a familify of standard Prym varieties. 
%%	 OLD
%%Furthermore, since $\gamma((-\epsilon, \epsilon))\subset A_g$ is real analytic, we can take a local holomorphic extension to a smooth complex curve $B\supseteq\gamma((-\epsilon, \epsilon))$ see Lemma \ref{prop:ciccia}). In particular $\Gamma_1=\gamma(-\epsilon, 0)\subset P_{g+1}$ and $\Gamma_2=\gamma(0,\epsilon)\subset P_{g+1}$ so if $B_1$ and $B_2$ denote the extension over $\Gamma_1$ and $\Gamma_2$ obtained by restriction from $B$, then I can find $B$ such that $B_1\cup B_2 \subset P_{g+1}$. Let $B^0=B\setminus (B_1\cup B_2)$. Then $B_0\subset B$ is the only locus where non-standard generalized Prym varieties can occur.

\vspace{1.5mm}
Let $\phi: \cP\to B$ be the associated family of Prym varieties (defined up to finite base change). By construction,  $P_b$ is a standard Prym variety for any $b\in B^0$ and  $P=\phi^{-1}(\gamma(0))=JC$ (i.e. it is a generalized Prym variety occurring as a Jacobian variety).
We consider $\mathcal{\pi}:\tilde{\cC}\to \cC'$, a family of coverings over $B$ naturally associated to it (defined up to finite base change by the Prym map $P$). The fibre for any $b\in B$ is a covering $\pi_b:\tilde{C}_b\to C'_b$ corresponding to a point $[C'_b, \eta_b]\in P(S)\subset \overline{R_{g+1}}$ with image $P([C'_b, \eta_b])=P_b$, the fibre of $\phi$ over $b\in B$. 

For any $b\in B^0,$ $[C'_b, \eta_b]\in R_{g+1}$ and so $ d P= d \Pr : H^1(C'_b, T_{C'_b})\to \Sym^2H^0(C'_b, \omega_{C'_b}\otimes \eta_b)^\vee$. Under the isomorphism $T_{[P_b]}A_g \simeq  \Sym^2H^0(C_b\otimes \eta_b)^\vee$, for any $\gamma'(t) \in T_{[P_b]}A_g,$ there is  $0\neq \xi_b\in H^1(C'_b, T_{C'_b})$ such that $$\gamma'(t)=d \Pr \xi_{\gamma(t)}=\cup\xi_{\gamma(t)}:H^0(C'_{\gamma(t)},\omega_{C'_{\gamma(t)}}\otimes \eta_{\gamma(t)})\to H^1(C'_{\gamma(t)},\eta_{\gamma(t)}).$$ Then, $\dim\ker \cup\xi_{\gamma(t)}=g-k.$ Indeed, for $[P_{\gamma(0)}]=[JC]\in \overline{J_g}$, using the Torelli map we have $dj(\xi_{\gamma(0)})=\cup \xi_{\gamma(0)}\in \Sym^2 H^1(C,\cO_C)\simeq T_{[JC]}A_g$. By Lemma \ref{lem:Jacobians}, $\dim \ker \cup \xi_{\gamma(0)}=g-k $ and so moving using parallel transport we first get  $\dim \ker \cup\xi_{\gamma(s)}=\dim \ker \cup \xi_{\gamma(0)}=g-k$ and then also for a general $b\in B$.
%%%\begin{remark} The morphism $\overline{\pi_{R}}: \overline{R_{g+1}}\to \overline{M_{g+1}}$ is ramified along the boundary and \'etale otherwise. For this reason, it is not useful to describe $\gamma'(0)=\cup \xi \in T_{[JC]}A_g$ in terms of its preimage under the differential of the Prym map at the generalized Prym variety $P=JC$, but it is instead on any other point of the family, where $	\overline{\pi}=\pi$ is \'etale and so we can then lift the tangent space from $M_{g+1}$ and do the computation.
%%%\end{remark}

By looking at $\xi_b\in H^1(C'_b,T_{C'_b})$ as an extension under the isomorphism $H^1(C'_b,T_{C'_b})\simeq \Ext^1(\cO_{C'_b},\omega_{C'_b})$, we consider two s.e.s.
\begin{equation}
\label{ses:firstextension}
\begin{tikzcd}
0 \arrow{r} & \cO_{C'_b} \arrow{r} & E'_b 
\arrow{r}& \om _{C'_b} \arrow{r} & 0 
&:  \xi_b, \\
%\end{tikzcd} 
%%\label{ses:firstextensiontwist}
%\begin{tikzcd}
0 \arrow{r} & \eta_b \arrow{r} & E'_b\otimes \eta_b
\arrow{r}& \om _{C'_b}\otimes \eta_b \arrow{r} & 0 &: \xi_b,
\end{tikzcd}
\end{equation}
both classified by $\xi_b$, obtained reciprocally by twisting with $\eta_b\in \Pic^0(C_b)_2$, under the isomorphisms $\Ext^1(\cO_{C'_b}, \omega_{C't})\simeq H^1(\cO_{C'_b}\otimes \omega^{\vee}_{C'_b})\simeq H^1(\cO_{C'_b}\otimes \omega^{\vee}_{C'_b}\otimes\eta_b^2)=H^1(\cO_{C'_b}\otimes\eta_b\otimes \omega^{\vee}_{C'_b}\otimes \eta^{-1}_b)\simeq \Ext^1(\eta_b, \omega_{C'_b}\otimes\eta_b)$. Focusing on the second s.e.s., the coboundary maps in the l.e.s. in cohomology is 
$$\cup\xi_b: H^0(C'_b, \omega_{C'_b}\otimes\eta_b)\to H^0(C'_b, \omega_{C'_b}\otimes\eta_b)^\vee,$$ under the isomorphism $H^1(C'_b,\eta_b)\simeq H^0(C'_b, \omega_{C'_b}\otimes\eta_b)^\vee$, and so it corresponds exactly to the tangent vector $\gamma'(t)\in T_{[P_{\gamma(t)}]}A_g$ at $\gamma(t)$, under the isomorphism $T_{[P_{\gamma(t)}]}A_g\simeq \Sym^2H^0(C_b\otimes \eta_b)^\vee$. Then as seen before, $\dim \ker \cup\xi_{\gamma(t)} =g-k$.
%%%\begin{remark}Notice that the coboundary morphism of the first ses $\cup\xi_t : H^0(C'_t,\omega_{C'_t})\to H^1(C'_t,\cO_{C'_t})$ corresponds to a vector in $T_{[JC'_t]}M_{g+1}$ and we do not have any information on its rank. \end{remark}
%%%
%%%	
%%%	\smallskip
%%%	\noindent

\smallskip
\noindent
We now consider the map $\alpha'_b:\bigwedge^2 H^0(C'_b, E'_b\otimes \eta_b)\to H^0(C'_b,\omega_{C'_b}),$ defined under the isomorphism $ H^0(\det E'_b\otimes \eta_b)\simeq H^0(C'_b,\omega_{C'_b}),$ and we use Lemma \ref{lem-decel} to construct a useful line bundle. To apply the lemma, we just have to check that $$2 h^0(C'_b,E'_b\otimes \eta_b)-3 > h^0(C'_b, \omega_{C'_b}).$$ We have that $h^0(C'_b,\omega_{C'_b})=g+1$ and $h^0(C'_b, E'_b\otimes \eta_b)=g-k$, because $h^0(C'_b,\eta_b)=0$, $\dim \ker \cup \xi_b=g-k$ and $h^0(C'_b, E'_b\otimes \eta_b)= h^0(C'_b,\eta_b)+ \dim \ker \cup \xi_b=g-k$. Substituting in the condition above we get $k< (g-4)/2$ and this is satisfied for any $k  <\Cliff C -2$ (because $\Cliff C \leq \left\lfloor \frac{g-1}{2}\right \rfloor$, see \ref{en:goncliff}, Section \ref{sec:acgc}), according to our assumptions. 
	Applying Lemma \ref{lem-decel}, we get a decomposable element in the kernel of $\alpha'_b$, defining in a natural way (up to saturation) a line bundle $\cM_b$ on $C'_b$ such that $h^0(\cM_b)\geq 2$ and a diagram 
\begin{equation}\label{dia:twistcliff}
\begin{tikzcd}
& & 0 \arrow{d} & & \\
& & \cM_b \arrow{d}{} & & \\
0 \arrow {r} & \eta_b \arrow{r}{}  & E'_b\otimes \eta_b
\arrow{d}{} \arrow{r} &
\om_{C'_b}\otimes\eta_b \arrow{r} & 0 \\
& & \omega_{C'_b}\otimes \cM_b^{-1}\arrow{d} & &\\
&& 0 &&
\end{tikzcd}
\end{equation}
The proof of the theorem at this point is developed by distinguishing three cases: $$h^0(C'_b, \omega_{C'_b}\otimes \cM_b^{-1})=\geq 2,0,1,$$ for $b$ general in $ B$. 

\smallskip 
{\bf {\em Case  1: $h^0(C'_b, \omega_{C'_b}\otimes \cM_b^{-1})\geq 2$.}} The case in not admitted by the assumption on the Clifford index of $C$. Indeed, since both $h^0(C'_b, \cM_b)$ and $h^0(C'_b, \omega_{C'_b}\otimes \cM_b^{-1})\geq 2$, the line bundle $\cM_b$ contributes to the Clifford index of $C'_b$. We estimate it comparing the followings
\begin{gather*} h^0(C'_b, \cM_b)+h^0(C'_b, \omega_{C'_b}\otimes \cM_b^{-1})\geq g-k \\ h^0(C'_b, \cM_b)-h^0(C'_b, \omega_{C'_b}\otimes \cM_b^{-1})=\deg \cM_b - (g+1)+1,\end{gather*} given respectively by Diagram \ref{dia:twistcliff} together with $h^0(C'_b, E'_b\otimes \eta_b)=g-k$  and by the Riemann Roch formula. We get $2h^0(\cM_b)\geq \deg \cM_b-k$ and then $\Cliff (C'_b)\leq\Cliff (\cM_b)\leq k+2$. But $b$ is general in $B$ and by Lemma \ref{lem:cliffnorm} we obtain $\Cliff C\leq \Cliff (C'_b)+1\leq k+3$, which is not admitted by assumption.
 
  \smallskip 
  \noindent

{\bf {\em	Case 2: $h^0(C'_b, \omega_{C'_b}\otimes \cM_b^{-1})=0$}.} We consider the family $\phi:\cP\to B$ of Prym varieties associated to $B\subset P_{g+1}$ (up to a finite base change) and let  $\cU\subset \phi_{\ast}\Omega^1_{\cP/B}$ be the the unitary flat bundle (Definition \ref{defUK} and Remark \ref{rem:extU}) with respect to the p.v.h.s. \ref{eq:vhsAV} (case $(1)$ and $(2)$) . In this case, $\cU$ has a very special geometric property: let  $U_b$ be the fibre over a general $b\in B$,
\begin{itemize}\item[$(P_{U})$] $ U_b = \ker \cup \xi_b$ is a isotropic vector space with respect to $\alpha'_b.$ 
	\end{itemize} 
	Indeed,  $\gamma((-\epsilon, \epsilon))\subset B$ is a geodesic and so for a general $b$, $ U_b= \ker \cup \xi_b \subset H^0(C'_b,\omega_{C'_b}\otimes \eta_b).$ By assumption $h^0(C'_b, \omega_{C'_b}\otimes \cM_b^{-1})=0$ and so $\ker \cup\xi_b=H^0(C'_b, E'_b\otimes\omega_{C'_b})=H^0(C'_b, \cM_b)$, which means by the last equality that it is a isotropic subspace with respect to $\alpha'_b$ for a general $b\in B$. This proves property $(P_U)$. 
	
	Now we want to use this fact and Theorem \ref{thm:cdfU} to deduce a very special geometric property on the rank of the unitary flat bundle $\tilde{\cU}$ associated to the family $\tilde{f}: \tilde{\cC}\to B$ of genus $2g+1$ curves constructed by $\phi:\cP\to B$ (up to finite base change). To be more precise, we recall that associated to the family $\phi:\cP\to B$ there is a family  $\tilde{f}: \tilde{\cC}\to B$ of genus $2g+1$ stable curves and a $B$-involution $i:\tilde{C}\to \tilde{C}$ inducing a family of coverings $\mathcal{\pi}:\tilde{\cC}\to \cC'=\tilde{\cC}/(i)$ and in particular also a family $f':\cC'\to B$ of genus $g+1$ curves. Denote by  $\tilde{\cU}\subset \tilde{f}_\ast \omega_{\tilde{\cC}/B}$ and $\cU'\subset f'_\ast \omega_{\cC'/B}$ their unitary flat bundles (Definition \ref{defUK} and Remark \ref{rem:extU}) with respect to the p.v.h.s. \eqref{eq:vhsAV} (case $(1)$ and $(2)$, respectively). By construction, since unitary flat, there is a splitting $\tilde{\cU}=\cU^+\oplus \cU^-$ into a $i-$invariant part $\cU^+$ and an $i$-antiinvariant part $\cU^-$ and in particular $\cU^- \simeq \cU'\simeq  \cU$. 
	
	We now consider s.e.s. associated to $\tilde{\xi}_b\in H^1(\tilde{C}_b,T_{\tilde{C}_b})\simeq \Ext^1(\cO_{\tilde{C}_b},\omega_{\tilde{C}_b})$.
	\begin{equation} \label{ses:firstextensionup}
	\begin{tikzcd}
	0 \arrow{r} & \cO_{\tilde{C}_b} \arrow{r} & \tilde{E}_b
	\arrow{r}& \om _{\tilde{C}_b} \arrow{r} & 0 &: \tilde{\xi}_b .
	\end{tikzcd}
	\end{equation}
	and the map $\tilde{\alpha}: \bigwedge^2(\tilde{C}_b, \tilde{E_b}\to H^0(\tilde{C}_b,\omega_{\tilde{C}_b})$. We want to prove that the fibres $U^-_b=U_b\subset H^0(\tilde{C}_b,\omega_{\tilde{C}_b})$ are isotropic  with respect to $\tilde{\alpha}_b$ for any $t\in B$.
	By construction, the coboundary morphism $\cup\tilde{\xi}: H^0(\tilde{C}_b, \omega_{\tilde{C}_b})\to H^1(\tilde{C}_b, \cO_{\tilde{C}_b})$ is given by $$\cup \xi_b : H^0(C'_b, \omega_{C'_b})\oplus H^0(C'_b,\omega_{C'_b}\otimes \eta_b)\to H^1(C'_b, \cO_{C'_b})\oplus H^1(C'_b,\cO_{C'_b}\otimes \eta_b),$$ where $\xi_b \in H^1(C'_b, T_{C'_b})$ acts diagonally on the two pieces as defined in \ref{ses:firstextension}, under the isomorphisms 
		$H^0(\tilde{C}_b, \cO_{\tilde{C}_b})\simeq H^0(C'_b, \cO_{C'_b})\oplus H^0(C'_b,\cO_{C'_b}\otimes \eta_b)$ and $H^0(\tilde{C}_b, \omega_{\tilde{C}_b})\simeq H^0(C'_b, \omega_{C'_b})\oplus H^0(C'_b,\omega_{C'_b}\otimes \eta_b)$. 
		%%In particular $\tilde{\xi}\in H^1(\tilde{C_b}, T_{\tilde{C_b}})=H^0(\tilde{C}_b, \omega^{\otimes 2}_{\tilde{C}_b})^{\vee}\simeq H^0(C'_b, \omega^{\otimes 2}_{C'_b})^\vee\oplus H^0(C'_b,\omega^{\otimes 2}_{C'_b}\otimes \eta_b)^\vee.$ We consider the s.e.s 
	 Then the fibre $U^{-}_b$ of $\cU^{-}$ is $$U^-_b=U_b=\ker \{\cup \xi_b: H^0(C'_b,\omega_{C'_b}\otimes \eta_b)\to H^1(C'_b,\cO_{C'_b}\otimes \eta_b)\}.$$ So $\cU=\cU^{-}$ and consequently it is isotropic with respect to the map $\tilde{\alpha_b}$ seen under the identifications above. Summing up, we can think at  $U_b\subset H^0(\tilde{C_b},\omega_{\tilde{C_b}})$ as an isotropic space and apply Theorem \ref{thm:cdfU}, 
%	use the lifting Lemma (see \cite{PT}) and the Castelnuovo de-Franchis theorem for tubular surfaces (\cite{GST}) on the family $\tilde{f}$ of $2g+1$-genus curves. 
%	By the Castelnuovo de-Franchis theorem for tubular surfaces and the lifting Lemma (see \cite{PT}, \cite{GST}), 
	getting a non constant morphism $\tilde{C_b}\to \Sigma$, for any $t\in B$, where $\Sigma$ is a curve of genus $g(\Sigma)$ equal to the rank of $\cU$ and so $g(\Sigma)=\rk \cU=\dim U_b=\dim \ker \cup \xi_b=g-k$. We now recall that for $t=\gamma(0)$, $\tilde{C}_{\gamma(0)}$ is a singular curve defining the generalized Prym variety with the associated involution. By Lemma \ref{lem:normJacPrym}, let $\nu_{\tilde{C}_{\gamma(0)}}:\tilde{C}_{\gamma(0)}^{\nu}\to \tilde{C}_{\gamma(0)}$ be the normalization map, then $C\subset \tilde{C}_{\gamma(0)}$. By using the map $\tilde{C_0}\to D$, we get by composition a non-constant map $C\to \Sigma$ of smooth projective curves of genus $g\geq 2$. But now $C$ has genus $g$, $\Sigma\ncong C$, and so by the Riemann Hurwits we have $\rk \cU=g(D)\leq (g+1)/2$. 
	
	To conclude the proof, we just notice that the property $\rk \cU\leq (g+1)/2$ is not admitted by our assumptions.. Indeed, by Lemma \ref{prop:ciccia} $\rk \cU =\rk \cK$ and $\rk \cK =\dim\ker \cdot \cup \xi_b=g-k$ so we get $(g+1)/2 \geq \rk\cU=g-k.$ Since by assumption $k<\Cliff C -3\leq \Cliff C,$ this is not possible.

 	  \smallskip 
 	  \noindent
 	  
 	{\bf {\em Case 3: $h^0(C'_b, \omega_{C'_b}\otimes \cM_b^{-1})=1$.} }  We first of all prove that in this case for any $b\in B^0$, the point of order two $\eta_b\in \Pic^0(C'_b)_2$ satisfies the following property 
 	\begin{itemize}\item[(P)]  $\eta_b=\cO_{C'_b}((s)-(s')),$ where $s' \in H^0(C'_b, \omega_{C'_b}\otimes \cM_b^{-1}),$ $s\in H^0(C'_b, \omega_{C'_b}\otimes \cM_b^{-1}\otimes \eta_b)$ and $\deg (s)=\deg (s')\geq k+2.$
 		\end{itemize}
 	
 	By assumption, $h^0(C'_b, \omega_{C'_b}\otimes \cM^{-1})=1,$ which means that there is a non zero $s'\in H^0(C'_b, \omega_{C'_b}\otimes \cM^{-1})$ and it is unique up to a multiplying by a non zero scalar so that we can reduce to work on $s$. By twisting Diagram \ref{dia:twistcliff} with $\eta_b$, we construct the following
 	\begin{equation}\label{dia:cliffCC}
 	\begin{tikzcd}
 	& & 0 \arrow{d} & & \\
 	& & \cM_b \otimes \eta_b \arrow{d}{i_2} & & \\
 	0 \arrow {r} & \cO_{C'} \arrow{r}{i} \arrow{rd}{s} & E'_b
 	\arrow{d}{p_{2}} \arrow{r} &
 	\om_{C'_b} \arrow{r} & 0 \\
 	& & \omega_{C'_b}\otimes \cM_b^{-1}\otimes\eta_b\arrow{d} & &\\
 	&& 0 &&
 	\end{tikzcd}
 	\end{equation}
 	We consider $s=p_2 \circ i\in H^0(C'_b, \omega_{C'_b}\otimes \cM_b^{-1}\otimes\eta_b)$ and we prove that $s$ is not the zero section. By contradiction, let $s\equiv 0$. Then $i(1)=i_2(\sigma)$, for some $\sigma \in H^0(\cM_b\otimes\eta_b)$ different from zero. So $\cM_b\otimes \eta_b\simeq \cO_{C'_b}$ and, since $\eta_b^2=\cO_{C'_b}$, we conclude that $\cM_b\simeq \eta_b$.  But now $h^0(C'_b,\cM_b)\geq 2,$ $h^0(C'_v, \eta_b)=0$ and $h^0(C'_b,\cM_b)=h^0(C'_b, \eta_b)=0$ leads to a contradiction. 
 	
% 	 \blue{check!}
 	\smallskip
 	\noindent
 	 Observe that  we can choose $s'\in H^0(C'_b, \omega_{C'_b}\otimes \cM^{-1}_b)$ such that $\cO_{C'_b}((s)-(s'))= \eta_b$. Then, we immediately have $s\neq s',$ because $\eta_b\neq \cO_{C'_b},$ and $\deg (s)= \deg (s'),$ because $\deg \eta_b=0$. We now prove that $\deg (s)\leq k+2$. For this, by by using Diagram \ref{dia:twistcliff} we have
 	$$g-k = h^0(C'_b,E'_b\otimes \eta_b)\geq h^0(C'_b, \cM_b)+h^0(C'_b, \omega_{C'_b}\otimes \cM_b^{-1})$$
 	and since by assumption $h^0(C'_b, \omega_{C'_b}\otimes \cM_b^{-1})=1$ we get 
 		\begin{equation}\label{eq:detM}  h^0(C'_b, \cM_b)\leq g-k-1 \quad\quad \mbox{and }\quad\quad \deg \cM_b\leq 2g-k-2.
 		\end{equation}
 	 and so $\deg (s)=\deg \omega_{C'_b}-\deg \cM_b\geq 2(g+1)-2-2g+k+2=k+2$ as wanted.
 
 	We now prove that the Jacobian variety $JC$ of a general curve $[C]\in M_g$ does not arise from the situation above. To do this, let $Z_{(P)}$ be the locus of curves $[C']\in \overline{M_{g+1}}$ admitting $\eta\in \Pic^0(C')_2$ satistying $(P)$. Let $\Delta_0\subset M_{g+1}$ be the locus of the irreducible $1-$nodal curves, defined as the closure of $\Delta_0^{0}\subset \Delta_0,$ the open Zariski dense of irreducible $1-$nodal curves. Consider the modular map 
 	$\phi^{0}: \Delta^{0}_0\to M_g$, sending any irreducible $1-$nodal curve $C'$ of genus $g+1$ to its normalization $\nu:C^{\nu}\to C'$, which is indeed a smooth curve of genus $g$. 
 	Let $W_{(P)}\subset M_g$ be the image of $\overline{Z_{(P)}}\cap \Delta^0_0$ under this map. By construction it is the locus of curves $C$, which are the normalization of a irreducible $1-$nodal curve $[C']\in \overline{Z_{(P)}},$ namely that is limit of a family of curves satisfying $(P)$. We prove that the codimension dimension of this locus is at least $1$. 
 	Let $\cH_{2d, 2d+r}$ be the Hurwitz space of maps of  degree $2d$  from a genus $(g+1)$-smooth projective curve to the projective space $\bP^1$ ramified in $2d+r$ points and let $\pi_{\cH}:\cH_{2d, 2d+r}\to M_{g+1}$ be the map sending $[f: C' \to \bP^1]\mapsto [C'].$ Consider 
 	$$ Y_d=\{[C']\in M_{g+1}\,|\, \exists M, N \in \Div (C') \mbox{ effective, } \deg M=\deg N=d\geq k+2, 2N\equiv 2M \}.$$ Then $Y_d\subset \cH_{2d, 2d+r'}$ is by definition a subvariety of maps with two branch points of ramification type fixed by $2N\equiv 2M$ and so $\dim Y_g=r'-1\leq 2g+2d-1$, because by the Riemann-Hurwitz formula any element of $Y_d$ must have $r'$ additional ramification points, with $r'\leq 2g+2d$, and furthermore since on $Y_d$ two fibres are fixed, then two branches are fixed and there is only one branch point movable by using the automorphisms of $\bP^1$. By definition $Z_{(P)}=\pi(Y_d)$ and so $\dim Z_{(P)}\leq 2g+2d-1$. 
% 	(\blue{because of some property on $\pi$ })
 	 By using the compactification $\overline{\pi}_{\cH}:\overline{\cH_{2d, 2d+r}}\to \overline{M_{g+1}}$ as in \cite{HM82}, the closure $\overline{Z_{(P)}}\subset \overline{M_{g+1}}$ is compatible $\overline{\pi}_{\cH}(\overline{Y_d})$, the image of the closure $\overline{Y_d}\subset \overline{\cH}_{2d, 2d+r},$ and the dimensions of the loci are preserved. So $\dim \overline{Z_{(P)}}\leq 2g+2d-1$ and $\dim \overline{Z_{(P)}}\cap \Delta^0_0\leq\dim\overline{Z_{(P)}}-1$. Summing up, if $\overline{Z_{(P)}}-1\leq 3g-4$, $W_{(P)}$ is of codimension at least one. By contradiction, if $2g+2d-1>3g-4$, then $2d> g-3$. By assumption $2k+4>2d>g-3$, so we conclude that $k>(g-7)/2$. But now $k> (g-1)/2- 3 \geq \Cliff C-3$ (using $\Cliff C \leq \left\lfloor \frac{g-1}{2}\right \rfloor$, see \ref{en:goncliff}, Section \ref{sec:acgc}), which leads to a contradiction by assumption.

	\bibliographystyle{alpha}
	\bibliography{mybib201?}

\begin{thebibliography}{ACGH85}

\bibitem[ACGH85]{ACGH}
Enrico Arbarello, Maurizio Cornalba, Philiph Griffiths, and Joe Harris.
\newblock {\em Geometry of algebraic curves. {V}ol. {I}}, volume 267 of {\em
  Grundlehren der Mathematischen Wissenschaften [Fundamental Principles of
  Mathematical Sciences]}.
\newblock Springer-Verlag, New York, 1985.

\bibitem[Bea77]{B77}
Arnaud Beauville.
\newblock Prym varieties and the {S}chottky problem.
\newblock {\em Invent. Math.}, 41(2):149--196, 1977.

\bibitem[Bea83]{Bea83}
Arnaud Beauville.
\newblock {\em Complex algebraic surfaces}, volume~68 of {\em London
  Mathematical Society Lecture Note Series}.
\newblock Cambridge University Press, Cambridge, 1983.
\newblock Translated from the French by R. Barlow, N. I. Shepherd-Barron and M.
  Reid.

\bibitem[Cap11]{C11}
L.~Caporaso.
\newblock Linear series on semistable curves.
\newblock {\em Int. Math. Res. Not.}, 13:2921--2969, 2011.

\bibitem[CD17]{CD:Answer_2017}
Fabrizio Catanese and Michael Dettweiler.
\newblock Answer to a question by {F}ujita on {V}ariation of {H}odge
  {S}tructures.
\newblock {\em Adv. Stud. in Pure Math.}, 74-04(04):73--102, 2017.

\bibitem[CF19]{CF19}
Elisabetta Colombo and Paola Frediani.
\newblock A bound on the dimension of a totally geodesic submanifold in the
  prym locus.
\newblock {\em Collectanea Mathematica}, 70(5.1):51--57, 2019.

\bibitem[CM91]{CM91}
Marc Coppens and Gerriet Martens.
\newblock Secant spaces and {C}lifford's theorem.
\newblock {\em Compositio Math.}, 78(2):193--212, 1991.

\bibitem[Cor85]{C85}
Maurizio Cornalba.
\newblock Syst\`emes pluricanoniques sur l'espace des modules des courbes et
  diviseurs de courbes {$k$}-gonales (d'apr\`es {H}arris et {M}umford).
\newblock {\em Ast\'{e}risque}, (121-122):7--24, 1985.
\newblock Seminar Bourbaki, Vol. 1983/84.

\bibitem[ELMS89]{ELMS89}
David Eisenbud, Herbert Lange, Gerriet Martens, and Frank-Olaf Schreyer.
\newblock The {C}lifford dimension of a projective curve.
\newblock {\em Compositio Math.}, 72(2):173--204, 1989.

\bibitem[Far]{F11}
Gavril Farkas.
\newblock Prym varieries and their moduli.
\newblock {\em To appear in "Contributions to algebraic geometry" edited by P.
  Pragacz and published by the EMS}.

\bibitem[Fuj78]{Fuj78b}
Takao Fujita.
\newblock The sheaf of relative canonical forms of a {K}\"ahler fiber space
  over a curve.
\newblock {\em Proc. Japan Acad. Ser. A Math. Sci.}, 54(7):183--184, 1978.

\bibitem[GAT18]{GT19}
V\'{\i}ctor Gonz\'alez-Alonso and Sara Torelli.
\newblock Families of curves with higgs field of arbitrarily large kernel.
\newblock {\em preprint}, 2018.

\bibitem[Ghi18]{G18}
Alessandro Ghigi.
\newblock On some differential-geometric aspects of the torelli map.
\newblock {\em To appear on Boll. Unione. Mat. Ital.}, (7), 2018.

\bibitem[GPT19]{GPT18}
Alessandro Ghigi, Gian~Pietro Pirola, and Sara Torelli.
\newblock Totally geodesic subvarieties in the moduli space of curves.
\newblock {\em preprint}, 2019.

\bibitem[Gri71]{G_ATrascendental_1971}
Phillip~A. Griffiths.
\newblock A transcendental method in algebraic geometry.
\newblock In {\em Actes du {C}ongr\`es {I}nternational des {M}ath\'ematiciens
  ({N}ice, 1970), {T}ome 1}, pages 113--119. Gauthier-Villars, Paris, 1971.

\bibitem[GST19]{GST}
V.~{Gonz{\'a}lez-Alonso}, L.~{Stoppino}, and S.~{Torelli}.
\newblock {On the rank of the flat unitary factor of the Hodge bundle}.
\newblock {\em To appear on Transaction of AMS}, 2019.

\bibitem[HM82]{HM82}
Joe Harris and David Mumford.
\newblock On the {K}odaira dimension of the moduli space of curves.
\newblock {\em Invent. Math.}, 67(1):23--88, 1982.
\newblock With an appendix by William Fulton.

\bibitem[KO19]{KO19}
Bruno Klingler and Anna Otwinowska.
\newblock On the closure of the positive hodge locus.
\newblock {\em preprint}, 2019.

\bibitem[Kon99]{K99}
K.~Konno.
\newblock Clifford index and the slope of fibered surfaces.
\newblock {\em J. Algebraic Geom.}, (8(2)):2'7--220, 1999.

\bibitem[Mum74]{M74}
David Mumford.
\newblock Prym varieties. {I}.
\newblock In {\em Contributions to analysis (a collection of papers dedicated
  to {L}ipman {B}ers)}, pages 325--350. 1974.

\bibitem[PT19]{PT}
Gian~Pietro Pirola and Sara Torelli.
\newblock Massey products and {F}ujita decompositions.
\newblock {\em Collectanea Matematica}, 2019.

\bibitem[Sho84]{Sho84}
V~V Shokurov.
\newblock {PRYM} {VARIETIES}: {THEORY} {AND} {APPLICATIONS}.
\newblock {\em Mathematics of the {USSR}-Izvestiya}, 23(1):83--147, feb 1984.

\end{thebibliography}

\end{document}